\documentclass[12pt,twoside,reqno]{amsart}

\usepackage[OT1]{fontenc}    %
\usepackage{type1cm}         %
\usepackage{amsthm}          %
\usepackage{amssymb}         
\usepackage[bf]{caption}     
\usepackage{psfrag}
\usepackage{setspace}

\DeclareMathOperator{\dimm}{dim_M}         
\DeclareMathOperator{\dimh}{dim_H}         
\DeclareMathOperator{\diam}{diam}	   
\DeclareMathOperator{\dist}{dist}	   
\newcommand{\HH}{\mathcal{H}}
\newcommand{\MM}{\mathcal{M}}

\newcommand{\R}{\mathbb{R}}
\newcommand{\N}{\mathbb{N}}
\newcommand{\G}{\mathbb{G}}
\newcommand{\roo}{\varrho}
\newcommand{\iii}{\mathtt{i}}
\newcommand{\jjj}{\mathtt{j}}
\newcommand{\kkk}{\mathtt{k}}
\newcommand{\hhh}{\mathtt{h}}
\newcommand{\eps}{\varepsilon}

\theoremstyle{plain}
\newtheorem{theorem}{Theorem}[section]
\newtheorem{corollary}[theorem]{Corollary}
\newtheorem{proposition}[theorem]{Proposition}
\newtheorem{lemma}[theorem]{Lemma}

\theoremstyle{definition}
\newtheorem{definition}[theorem]{Definition}

\theoremstyle{remark}
\newtheorem{remark}[theorem]{Remark}
\newtheorem*{remark*}{Remark}
\newtheorem{question}[theorem]{Question}
\newtheorem{remarks}[theorem]{Remarks}
\newtheorem{example}[theorem]{Example}
\newtheorem*{acknowledgements}{Acknowledgements}

\numberwithin{equation}{section}

\begin{document}

\title[Moran constructions in metric spaces]{Weakly controlled Moran constructions and iterated functions systems in metric spaces}

\author{ Tapio Rajala  \and   
         Markku Vilppolainen      }
\address{Department of Mathematics and Statistics \\
         P.O. Box 35 (MaD) \\
         FI-40014 University of Jyv\"askyl\"a \\
         Finland}

\email{tapio.m.rajala@jyu.fi}
\email{markku.vilppolainen@jyu.fi}

\thanks{TR acknowledges the support of the Academy of Finland, project No. 211229.}
\date{\today}


\begin{abstract}
We study the Hausdorff measures of limit sets of weakly controlled Moran constructions in metric spaces. The separation of the construction pieces is closely related to the Hausdorff measure of the corresponding limit set. In particular, we investigate different separation conditions for semiconformal iterated function systems. Our work generalizes well known results on self-similar sets in metric spaces as well as results on controlled Moran constructions in Euclidean spaces.
\end{abstract}


\maketitle


\section{Introduction}

A familiar method of producing sets with fractal properties, such as the Cantor ternary set, is to start with a single compact subset of a metric space and proceed iteratively from one level of construction to the next by replacing each construction piece by  a fixed number of its compact subsets. The principal object of study, the \emph{limit set}, is then the set of those points from the start which do not get deleted in the process. 
Honoring  
the seminal contribution of P.\ A.\ P.\ Moran, who in \cite{M1946} initiated the study of sets which are nowadays called Moran fractals, cf. \cite{CM1992},  we call such a construction scheme a \emph{Moran construction}. It is evident that one needs to apply some control over the shapes and sizes of the construction pieces to get a manageable limit set.
 Like Moran, we are primarily interested in determining the Hausdorff dimension of the limit set and finding out whether or not the set has positive and/or finite Hausdorff measure in this dimension. We could go further and ask for the exact Hausdorff measure of the set. However, this question is very hard even for self-similar sets in $\R^n$. See for example \cite{ZF2004, M2005, LM2007}.



Fractal sets have traditionally been studied with the help of constructing functions. In particular, self-similar sets are constructed by iterating similitude mappings (which are shape-preserving by definition), see J. E. Hutchinson \cite{H1981}. Conditions like the open set condition and the strong open set condition have been invented to guarantee that the dimension of a self-similar set is fully determined by the contraction ratios of the constructing functions. 
Similar conditions are also available for Moran constructions. 
Two basic ones, the \emph{finite clustering property} and the \emph{ball condition}, were studied in detail in \cite{KV2008} 
in a Euclidean setting, see also \cite{K2004}. These conditions limit the amount of overlap between construction pieces. Likewise, under the (strong) open set condition, a self-similar subset of a Euclidean space is made up of its scaled-down copies with insignificant overlap between the parts. Accordingly, the aforementioned conditions will be referred to as \emph{separation conditions}. This paper in large part studies these conditions in the setting of general metric spaces.

The separation conditions are sometimes exactly what is needed for a self-similar set to have positive (and finite) Hausdorff measure at the expected dimension, see \cite{BG1992, S1994, S1996, K1995, BR2007}. 
We will however see that care must be taken with the choice of the class of functions when working in a general metric space. Example \ref{ex:OSCnoSOSC} gives a self-similar set in a complete doubling metric space for which the open set condition is satisfied, yet the dimension of this set 
cannot be inferred from the contraction ratios of the associated mappings. This contrasts the Euclidean case drastically. The set in the example is constructed with non-bijective similitudes.

To avoid examples like the one mentioned above, we define \emph{properly semiconformal iterated function systems} and prove for them (in the setting of doubling metric spaces) in Theorem \ref{thm:equiv} the equivalence between different separation conditions and positivity of the Hausdorff measure of the limit set at the critical dimension. A self-similar set constructed with bijective similitudes serves as a basic example for a limit set of a properly semiconformal iterated function system. Therefore, Theorem \ref{thm:equiv} generalizes \cite[Theorem 3.1]{BR2007}.

The paper is organized as follows. We begin Section \ref{section:preliminaries} by introducing the basic notation and recalling some definitions. Among these are the notions of controlled Moran construction and weakly controlled Moran construction. The rest of the section deals with basic properties of the topological pressure and the symbol space.

In Section \ref{section:separation}, we study the relationship between the basic separation conditions for Moran constructions and the Hausdorff dimension and measure of the limit set. We also investigate under what circumstances
the finite clustering property and the ball condition are actually equivalent. We mainly focus on doubling metric spaces and transfer as many of the results obtained in Euclidean spaces to the doubling metric spaces as possible. 

Section \ref{section:semiconformal} is devoted to semiconformal iterated function systems in metric spaces. The main focus in this section is the role of various separation conditions for semiconformal and self-similar iterated function systems. We establish the connection between Hausdorff measure, ball condition, open set condition and strong open set condition for properly semiconformal iterated functions systems in doubling metric spaces. The study of this connection was suggested for example in \cite{BR2007}. We also give some results for semiconformal iterated functions systems in non-doubling metric spaces.

In the final section, Section \ref{section:sub}, we define controlled sub-constructions of Moran constructions. We give examples of sub-constructions in Carnot groups which answer a question posed in \cite{BTW2009}.

\begin{acknowledgements} 
We thank Antti K\"aenm\"aki for the inspiring conversations 
during the preparation of this work and his valuable comments for the manuscript.
\end{acknowledgements}

\section{Notation and preliminaries}\label{section:preliminaries}

Let $(M,d)$ denote a metric space $M$ equipped with a metric $d$. We define an open ball to be $B(x,r) := \{y \in M : d(y,x)<r\}$. The diameter of a set $E \subset M$ is written as $\diam(E) := \sup\{d(x,y): x,y\in E\}$. The distance between two sets $E, F \subset M$ is denoted by $\dist(E,F) := \inf\{d(x,y) : x \in E, y \in F\}$. We also abbreviate $\dist(x,F) := \dist(\{x\},F)$. 

We will focus mainly on the Hausdorff dimension and measures of sets. Let $0 < s < \infty$ and $E \subset M$. The \emph{$s$-dimensional Hausdorff measure} of $E$ is defined as
\begin{align*}
\HH^s(E) := \lim_{\delta \to 0}\inf\biggl\{\sum_{i=1}^\infty\diam(A_i)^s : \;& E \subset \bigcup_{i=1}^\infty A_i \text{ and }\diam(A_i)<\delta\\
 &\text{ for every } i \in \N\biggr\}.
\end{align*}
The $0$-dimensional Hausdorff measure $\HH^{0}$ is defined to be the counting measure: $\HH^{0}(E) = \#E$.
The \emph{Hausdorff dimension} of a set $E \subset M$ is 
\[
\dimh(E) := \inf\{s : \HH^s(E) = 0\} = \sup\{s : \HH^s(E) = \infty\}.
\]
Another dimension we consider is the \emph{(upper) Minkowski dimension}, which is defined for a compact set $E \subset M$ as
\[
\dimm(E) := \limsup_{r \downarrow 0}\frac{-\log N(E,r)}{\log r}, 
\]
where
\[
 N(E,r) := \min\Big\{k : A \subset \bigcup_{i=1}^k B(x_i,r)\Big\}.
\]

In the set constructions of this paper we will always use 
an index set $I$ with $2 \le \#I < \infty$. The set of finite words will be written as $I^* := \bigcup_{n=1}^\infty I^n$. The set of infinite words is $I^\infty := I^\N$. For every word $\iii=(i_1, \dots, i_n) \in I^*$ we write the \emph{length} as $|\iii| = n$. With $\iii \in I^*$ and $\jjj \in I^*\cup I^\infty$ we write $\iii\jjj$ to mean the element in $I^*\cup I^\infty$ obtained by juxtaposing $\iii$ and $\jjj$. 

For $\iii \in I^*$ and $A \subset I^* \cup I^\infty$ we write $[\iii; A] := \{\iii\jjj : \jjj \in A\}$. With this we define the \emph{cylinder set} of $\iii \in I^*$ to be $\left[\iii\right] := [\iii; I^\infty]$. For $\iii \in I^*\cup I^\infty$ let $\iii|_n \in I^n$, with $1 \le n < |\iii|$, be so that $\left[\iii\right] \subset \left[\iii|_n\right].$ The notation $\iii \perp \jjj$ means that $\iii, \jjj \in I^*$ are \emph{incomparable}, that is, $\left[\iii\right] \cap \left[\jjj\right] = \emptyset$. For $\iii \in I^*$ we denote $\iii^- := \iii|_{|\iii|-1}$.

Recall that $I^\infty$ is a compact (ultra)metric space when equipped with the metric
\[
  d_2(\iii,\jjj) = \begin{cases}
                     2^{1-\min\{ k \,:\; \iii|_k \ne \jjj|_k\}} & \text{if } \, \iii \ne \jjj \\
                     0 & \text{if } \, \iii = \jjj
                   \end{cases}.
\]
In the \emph{symbol space} $(I^\infty, d_2)$ the balls are exactly the cylinder sets. Moreover, every cylinder has empty boundary.

\begin{definition} 
Let $M$ be a metric space. A collection $\{X_\iii : \iii \in I^*\}$ of compact subsets of $M$ with positive diameter is a \emph{weakly controlled Moran construction (WCMC)} provided that there exists a constant $D \ge 1$ so that for every $\iii, \jjj \in I^*$ the following four conditions hold:
\begin{enumerate}
  \renewcommand{\labelenumi}{(W\arabic{enumi})}
  \renewcommand{\theenumi}{(W\arabic{enumi})}
 \item \label{cond:incl}$X_{\iii} \subset X_{{\iii^-}}$,
 \item \label{cond:decr}there exists $n \in \N$ such that
 \[
  \max_{\iii\in I^n}\diam(X_\iii) < D^{-1},
 \]
 \item \label{cond:cover}$\diam(X_{\iii\jjj}) \le D\diam(X_\iii)\diam(X_\jjj)$,
 \item \label{cond:doub}$\diam(X_\iii) \ge D^{-1}\diam(X_{\iii^-})$.
\end{enumerate}
\end{definition}
WCMC is a generalization of the notion termed \emph{controlled Moran construction (CMC)} in \cite{KV2008}. In the definition of a controlled Moran construction we likewise use an indexed collection of compact sets and require that \ref{cond:incl} and \ref{cond:decr} are satisfied. Instead of conditions \ref{cond:cover} and \ref{cond:doub}, we assume the following stronger condition:
\begin{enumerate}
  \renewcommand{\labelenumi}{(C\arabic{enumi})}
  \renewcommand{\theenumi}{(C\arabic{enumi})}
 \item \label{cond:M2} for every $\iii, \jjj \in I^*$ we have
\[
 D^{-1} \le \frac{\diam(X_{\iii\jjj})}{\diam(X_\iii)\diam(X_\jjj)} \le D.
\]
\end{enumerate}

The next simple lemma is useful in many computations. For its proof, see \cite[Lemma 3.1]{KV2008}.
\begin{lemma}\label{lem:diamBoundExponential}
For a weakly controlled Moran construction there exist constants $c>0$ and $0<\roo<1$ so that
\begin{equation}\label{eq:geomdecr}
\diam(X_\iii) \le c \roo^{|\iii|}
\end{equation}
for every $\iii \in I^*$.
\end{lemma}

\begin{remark}
Assume that we have a CMC. From \ref{cond:M2} we get 
\[
\diam(X_\iii) \ge D^{-1} \diam(X_{\iii^-})\min_{j \in I}\diam(X_j) 
\]
and so the condition \ref{cond:doub} is satisfied. Therefore every CMC is a WCMC.
\end{remark}

Next we look at the basic properties of weakly controlled Moran constructions. Later on, we will assume more structure for the metric space and different separation conditions for the sets $X_\iii$. This section, however, deals only with results which hold in general.

Define a \emph{projection mapping} $\pi \colon I^\infty \to X$ by setting
\[
\{\pi(\iii)\} := \bigcap_{n = 1}^\infty X_{\iii|_n}
\]
for every $\iii \in I^\infty$. The intersection is non-empty because the sets $X_\iii$ are compact. The set $\pi(I^\infty)$ is called the \emph{limit set} of the WCMC. The usual candidate for the Hausdorff dimension of the limit set $E$ of a WCMC is the zero of the \emph{topological pressure} $P$ given by
\[
P(t) := \lim_{n \to \infty} \frac{1}{n}\log\sum_{\iii \in I^n}\diam(X_\iii)^t
\]
for each $t \ge 0$.
The existence of the defining limit follows by standard arguments from the theory of subadditive sequences.

The topological pressure is a convex function from the interval $\left[0, \infty\right)$ to $\R$ and is therefore automatically continuous outside the point $0$. To see the continuity at $0$ estimate using \ref{cond:doub}
\begin{align*}
 P(t) & = \lim_{n\to \infty}\frac{1}{n}\log\sum_{\iii\in I^n}\diam(X_\iii)^t \ge \lim_{n\to\infty} \frac{1}{n}\log(\sum_{\iii\in I^n}\diam(X_{\iii|_0})^t D^{-nt})\\
& \ge \lim_{n\to\infty} \frac{1}{n}\left(\log \#I^n + \log(\diam(X_{\iii|_0})^t D^{-nt})\right) = P(0) -t \log D \to P(0)
\end{align*}
as $t \to 0$. From the continuity it follows that there is always $t \ge 0$ so that $P(t) = 0$.

\begin{remark}
The condition \ref{cond:doub} is essential for the existence of the zero of the topological pressure. Consider an example with $I = \{1,2\}$ and $\diam(X_\iii) = 2^{-n^2}$ for $\iii \in I^n$. Now $P(0) = \log 2$, but for $t>0$
\[
 P(t) = \lim_{n\to\infty}\frac{1}{n}\log 2^{n-tn^2} = \lim_{n\to\infty}(1-tn) \log 2 = -\infty.
\]
\end{remark}

When using the topological pressure in the proofs, we usually need to move slightly away from the zero of the topological pressure. For doing this we need to observe that the topological pressure is strictly monotone.
\begin{lemma}\label{lemma:strictdec}
Assume that we have a WCMC. Then for $0\le s<t$ we have $P(t) < P(s)$.
\end{lemma}
\begin{proof}
Using \eqref{eq:geomdecr} we get
\begin{align*}
P(t) &= \lim_{n\to \infty} \frac{1}{n} \log\sum_{\iii\in I^n}\diam(X_\iii)^t\\
&\le \lim_{n \to \infty} \frac{1}{n}\left(\log\sum_{\iii\in I^n}\diam(X_\iii)^s + (t-s)\log\big(\max_{\jjj \in I^n}\diam(X_\jjj)\big)\right)\\
&\le P(s) + (t-s) \log \roo  < P(s)
\end{align*}
as claimed.
\end{proof}

Let us put Lemma \ref{lemma:strictdec} in use by proving an estimate for the Minkowski dimension of the limit set of a WCMC from its topological pressure.

\begin{proposition}\label{prop:UpperBoundForDimm}
 If the topological pressure of a WCMC satisfies $P(t) \le 0$ for a given $t \ge 0$, then we have $\dimm(E) \le t$.
\end{proposition}
\begin{proof}
 Take $s > t$. From Lemma \ref{lemma:strictdec} we see that $P(s)<0$. Therefore there exist $c<0$ and $n_0 \in \mathbb{N}$ so that
\[
 \frac{1}{n}\log\sum_{\iii \in I^n}\diam(X_\iii)^s < c
\]
with every $n \ge n_0$. Thus,
\[
\sum_{\iii \in I^n}\diam(X_\iii)^s < e^{cn}.
\]
By the repeated use of condition \ref{cond:doub} we see that for any $\iii \in I^*$
\begin{equation}\label{eq:growthlb}
 \diam(X_\iii) \ge D^{-|\iii|+1}\min_{j \in I}\diam(X_j).
\end{equation}
Now given $0 < r \le \min_{j \in I}\diam(X_j)$ define 
\[
n_r := \max\{n \in \N : \diam(X_\iii) \ge r \text{ for every } \iii \in I^n\}.
\]
Then by \eqref{eq:growthlb} we have
\[
N(E,r)r^s \le \sum_{\iii \in I^{n_r}}r^s \le \sum_{\iii \in I^{n_r}}\diam(X_\iii)^s \le e^{cn_r} \to 0
\]
as $r \downarrow 0$ and, consequently $\dimm(E) \le s$.
\end{proof}

A useful tool for studying the dimension of the limit set of WCMC is the following collection of measures $M^\psi$ which we obtain by using the well known Carath\'eodory's construction. Let $\psi \colon I^* \to \left[0,\infty\right[$ be a mapping such that $\max\{\psi(\iii) : \iii \in I^n\}\to 0$ as $n \to \infty$. Define for every $A \subset I^\infty$
\[
 M_n^\psi(A) := \inf\Big\{\sum_{\iii\in C}\psi(\iii) : C \subset I^*, A \subset \bigcup_{\iii \in C}[\iii], |\iii| \ge n\Big\}
\]
and from this
\[
 M^\psi(A) := \lim_{n \to \infty} M_n^\psi(A).
\]
In the case $\psi(\iii) = \diam(X_\iii)^t$ we write $M^t = M^\psi$. Notice that although the measures $M^t$ look like Hausdorff measures, they live on the symbol space $I^\infty$ and, without any separation condition for the sets $X_\iii$, they can not necessarily be pushed to be Hausdorff measures on a subset of the actual metric space $M$. 

On the symbol space  we have the following connection between the topological pressure and the measures $M^t$.
\begin{lemma}\label{lemma:1}
Given a WCMC and any $t\ge 0$ satisfying $P(t) \ge 0$, we have $M^t(I^\infty) > 0$. 
\end{lemma}
\begin{proof}
Assume, to the contrary, that $M^t(I^\infty) = 0$. Because $I^\infty$ is compact, there exists a finite set $Q \subset I^*$ and $s < t$ such that $I^\infty \subset \bigcup_{\iii \in Q}[\iii]$ and
\[
\sum_{\iii \in Q} \diam(X_\iii)^s < (2D^s)^{-1}.
\]
Therefore from \ref{cond:cover} we get
\begin{align*}
\sum_{\iii \in Q^*} \diam(X_\iii)^s & = \sum_{n =1}^\infty\sum_{\iii \in Q^n} \diam(X_\iii)^s 
\le \sum_{n =1}^\infty \left(\sum_{\iii \in Q}  D^s \diam(X_\iii)^s\right)^n\\ 
& \le \sum_{n =1}^\infty 2^{-n} = 1. 
\end{align*}
Denote $q = \max\{|\iii|:\iii \in Q\}$.
If now $\iii\in I^*$, there exists $\jjj \in Q^*$ and $\kkk \in I^*$ with $|\kkk| \le q$ so that $\iii\kkk = \jjj$. Hence for any $n \ge 1$ we get by using \ref{cond:doub}
\[
\sum_{\iii \in I^n}\diam(X_\iii)^s \le D^{qs}\sum_{\jjj\in Q^*}\diam(X_\jjj)^s \le D^{qs}.
\]
Thus $P(s) \le 0$. Because $P(t) < P(s)$ by Lemma \ref{lemma:strictdec}, we have arrived at a contradiction.
\end{proof}

The transition from the measure $M^t$ to a more suitable Borel measure $\mu$ will be done with the following version of Frostman's lemma. Regarding the proof, the idea of using standard techniques from functional analysis is due to J. D. Howroyd, see \cite[Theorem 2]{H1995}. The main part of the proof presented here is quite analogous to the proof of Frostman's lemma for standard Hausdorff measures given in \cite[Theorem 8.17]{M1995}. However, we give the details for the benefit of the reader.

\begin{proposition}\label{prop:Frostman}
Let $\psi \colon I^* \to [0,\infty)$ be a mapping such that $\max\{\psi(\iii) : \iii \in I^n\}\to 0$ as $n \to \infty$. Given $A \subset I^\infty$ with $M^\psi(A) > 0$, there exists an integer $n_0$ and a Borel measure $\mu$ on $I^\infty$ such that $0 < \mu(A)<\infty$ and
\[
 \mu(\left[\iii\right]) \le \psi(\iii)
\]
for every $\iii \in I^*$ with $|\iii| \ge n_0$. If $\psi(\iii) > 0$ for every $\iii \in I^*$, we can choose $n_0=1$.
\end{proposition}
\begin{proof}
Let $n_0\in \N$ be so large that $M_{n_0}^\psi(A) > 0$. Define a function $p$ on $C(I^\infty)$ (the space of continuous real-valued functions on $I^\infty$) by
\[
 p(f) = \inf \sum_\iii c_\iii \psi(\iii)
\]
where the infimum is taken over all finite or countable families $\{(\iii,c_\iii)\}$ such that $0 < c_\iii < \infty$, $|\iii|\ge n_0$ and
\[
 f|_A \le \sum_\iii c_\iii \chi_{[\iii]}.
\]
For $f, g \in C(I^\infty)$ and $t \ge 0$ we have $p(t f) = t p(f)$ and $p(f+g) \le p(f) + p(g)$. Let $\mathbf{1}$ denote the constant function from $I^\infty$ to reals with $\mathbf{1}(I^\infty) = \{1\}$. By the Hahn-Banach theorem (in the form presented e.g. in \cite[Theorem 3.2]{R1973}), we can extend the linear functional $c\mathbf{1} \mapsto c p(\mathbf{1})$, $c \in \R$, from the subspace of constant functions to a linear functional $L \colon C(I^\infty) \to \R$ satisfying $L(\mathbf{1}) = p(\mathbf{1}) = M_{n_0}^\psi(A)$ and
\[
 -p(-f) \le L(f) \le p(f) \text{ for } f \in C(I^\infty).
\]
If $f \ge 0$, $p(-f) = 0$ and so $L(f) \ge 0$. Hence by the Riesz representation theorem there exists a Borel measure $\mu$ on $I^\infty$ such that $L(f) = \int f d\mu$ for $f \in C(I^\infty)$. Because $\chi_{[\iii]} \in C(I^\infty)$ for every $\iii \in I^*$ we have
\[
 \mu([\iii]) = \int \chi_{[\iii]}(x)\,d\mu(x) = L(\chi_{[\iii]}) \le p(\chi_{[\iii]}) \le \psi(\iii)
\]
when $|\iii|\ge n_0$. Also $\mu(A) = M_{n_0}^\psi(A)$, which is clearly positive and finite.

Now assume that $\psi(\iii) > 0$ for each $\iii \in I^*$. To finish the proof, it suffices to show that $M_1^\psi(A) > 0$. We do this by assuming the contrary and deriving a contradiction. Let $\varepsilon_n = \min_{|\iii| \le n}\psi(\iii)$ for $n\in\N$. If $M_1^\psi(A) = 0$ then for each $n\in\N$ there exists $C_n \subset I^*$ such that $A \subset \bigcup_{\iii \in C_n} [\iii]$ and $\sum_{\iii\in C_n} \psi(\iii) < \varepsilon_n$. But then $\psi(\iii) < \varepsilon_n$ for each $\iii \in C_n$, so that $|\iii| > n$ for every $\iii \in C_n$ which draws us to conclude that
\[
  M_n^\psi(A) \le \sum_{\iii\in C_n} \psi(\iii) < \varepsilon_n \xrightarrow{n\to\infty} 0.
\]
This however contradicts the positivity of $M^\psi(A)$. Hence $M_1^\psi(A) > 0$.
\end{proof}

\section{Separation conditions for Moran constructions}\label{section:separation}

With any $n\in \N$ the sets $X_\iii$, $\iii \in I^n$, of a WCMC can have very different diameters. Therefore we define for $r > 0$
\[
 Z(r) := \{\iii \in I^* : \diam(X_\iii) \le r < \diam(X_{\iii^-})\}.
\]
Then each $X_\iii$ with $\iii \in Z(r)$ is a set of roughly diameter $r$. Also notice that $\iii \perp \jjj$ for
two distinct $\iii,\jjj \in Z(r)$. We define a local version of this for every $r>0$ and $x \in E$ as
\[
 Z(x,r) := \{\iii \in Z(r) : X_\iii \cap B(x,r) \ne \emptyset\}.
\]

Now we are ready to pass to the actual metric space and look for conditions on the sets $X_\iii$ which imply estimates on the Hausdorff measures.
A WCMC has the \emph{finite clustering property} if
\[
\sup_{x \in E}\limsup_{r \downarrow 0}\#Z(x,r) < \infty.
\]
This property is a sufficient separation condition to guarantee the positivity of the Hausdorff measure of the limit set of a WCMC.

\begin{proposition}\label{prop:fctoposfinite}
 Assume that for a WCMC the finite clustering property holds and $P(t) \ge 0$. Then $\mathcal{H}^t(E)>0$. Moreover, $\HH^t(E) < \infty$ if and only if $M^t(I^\infty) < \infty$.
\end{proposition}

\begin{proof}
 Because $P(t) \ge 0$, Lemma \ref{lemma:1} gives $M^t(I^\infty) > 0$. Therefore from Proposition \ref{prop:Frostman} we see that there exists a Borel measure $\mu$ on $I^\infty$ such that $0 < \mu(I^\infty)<\infty$ and
\[
 \mu([\iii]) \le \diam(X_\iii)^t
\]
for every $\iii \in I^*$. Let
\[
K = \sup_{x \in E}\limsup_{r \downarrow 0}\#Z(x,r).
\]
Take $k \in \N$ and define
\[
E_k = \{x \in E : \#Z(x,r) \le K \text{ for every } 0 < r < \frac{1}{k}\}.
\]
Choose any collection of sets $A_i \subset M$, $i \in \N$, for which $\diam(A_i)< \frac{1}{k}$, $A_i\cap E_k \ne \emptyset$ and $E_k \subset \bigcup_{i=1}^\infty A_i$. Fix for each $i \in \N$ a point $x_i \in A_i \cap E_k$. Now we can estimate
\begin{align*}
\mu \circ \pi^{-1}(E_k) & \le \sum_{i=1}^\infty\mu \circ \pi^{-1}(B(x_i,\diam(A_i)))\\
& \le \sum_{i=1}^\infty\sum_{\iii \in Z(x_i, \diam(A_i))} \mu(\left[\iii\right]) \le \sum_{i=1}^\infty K \diam(A_i)^t.
\end{align*}
Therefore by letting $k \to \infty$ we get
\[
\HH^t(E) \ge \frac{1}{K} \mu(I^\infty) > 0
\]
and the first claim is proved.

Suppose $M^t(I^\infty) < \infty$. Because of \eqref{eq:geomdecr} the sets $X_\iii$ serve as covering sets $A_i$ when calculating the Hausdorff measure. Hence $\HH^t(E) < \infty$.

Assume then that $\HH^t(E) < \infty$. Take $n \in \N$ and let $k \in \N$ be so large that $\frac{1}{k} < \diam(X_\iii)$ for every $\iii \in I^n$.
Define $E_k$ as before and take any collection of sets $A_i \subset M$, $i \in \N$, for which $\diam(A_i)< \frac{1}{k}$, $A_i\cap E_k \ne \emptyset$ and $E_k \subset \bigcup_{i=1}^\infty A_i$. Choose for each $i \in \N$ a point $x_i \in A_i \cap E_k$. Now 
\[
\pi^{-1}(E_k) \subset \bigcup_{\substack{i\in \N\\\iii \in Z(x_i, \diam(A_i))}}\left[\iii\right].
\]
Since $\diam(X_\iii) \le \diam(A_i) < \frac{1}{k}$ for $\iii \in Z(x_i, \diam(A_i))$ and thus $|\iii| \ge n$, we have
\[
M_n^t(\pi^{-1}(E_k)) \le K \sum_{i=1}^\infty \diam(A_i)^t.
\]
Therefore $M^t(\pi^{-1}(E_k)) \le K \HH^t(E_k)$. Because $E_{k_1} \subset E_{k_2}$ for $0 < k_2 < k_1$, we get
\[
M^t(I^\infty) = \lim_{k \to \infty} M^t(\pi^{-1}(E_k)) < K\HH^t(E)
\]
which completes the proof.
\end{proof}

By combining Proposition \ref{prop:fctoposfinite} with Proposition \ref{prop:UpperBoundForDimm} we see that 
\[
\dimh(E) = \dimm(E) = P^{-1}(0)
\]
for the limit set $E$ of a WCMC with the finite clustering property. In the Euclidean case this follows alternatively from a result by L. Barreira. Although in \cite[Theorem 2.1]{B1996} he assumed a stronger separation condition, he only needed the finite clustering property for the construction in the proof of \cite[Theorem 2.1(b)]{B1996}.

For the limit set $E$ of a WCMC it is not generally true that $\HH^t(E) < \infty$ when $P(t)=0$. This can be seen from the following example.

\begin{example}\label{ex:superCantor}
Take $I=\{1,2\}$ and define $\diam(X_\iii) = \frac{1}{2}$ for $\iii \in I$ and $\diam(X_\iii) = 2^{-2+\frac{1}{n}}\diam(X_{\iii^-})$ for $\iii \in I^*\setminus I$. Now
\[
 P(t) = \lim_{n \to \infty}\frac{1}{n}\Big(n-t\Big(2n-\sum_{k=1}^n\frac{1}{k}\Big)\Big)\log 2 = (1-2t)\log 2
\]
Therefore $P(\frac{1}{2}) = 0$. On the other hand, one can construct a Cantor set $E$ on $\R$ using such construction pieces to obtain $\HH^\frac{1}{2}(E) = \infty$, see \cite{QRS2001} for an exact formula for the Hausdorff measure of such constructions. Notice that the construction given there has the finite clustering property.
\end{example}

\begin{remarks}
(i) With a proof similar to that of Proposition \ref{prop:fctoposfinite} we can improve a result \cite[Theorem 5.1]{KV2009} on sub-self-affine sets. Namely, for a tractable sub-self-affine set $E_K$ in $\R^n$ having the finite clustering property we have $\HH^s(E_K) > 0$ when $P_K(s)=0$. (See \cite{KV2009} for the definition of a tractable sub-self-affine set.) Previously it was shown in \cite{KV2009} that $\dimh(E_K) = \dimm(E_K) = s$.

Self-affinity means that the constructing sets $X_\iii$ are obtained by iterating affine mappings $\{f_1, \dots, f_N\}$. The compact set $K \subset I^\infty$, referring to the prefix \emph{sub}, is assumed to be such that for every $(i_1, i_2, \dots) \in K$ also $(i_2, i_3, \dots) \in K$. The sub-self-affine set is defined as $E_K = \pi(K)$. Tractability for the sub-self-affine set is a condition which guarantees that the diameters of the constructing sets are comparable to the largest singular values of the constructing affine mappings, see \cite[Lemma A.3]{KS2009}.

In this setting the proof of Lemma \ref{lemma:1} can still be carried out and using $A = K$ in Proposition \ref{prop:Frostman} gives a measure $\mu$ on $K$ such that $0<\mu(K)<\infty$ and $\mu([\iii|_n]) \le \diam(X_{\iii|_n})^s$ for every $\iii \in K$ and $n \in \N$. The improvement is then finished with a similar use of the finite clustering property as in Proposition \ref{prop:fctoposfinite}.

(ii) Self-affine constructions are an important subclass of WCMC. They can have, similarily to the Example \ref{ex:superCantor}, $\HH^t(E)= \infty$ when $P(t) = 0$. Take for example a set in $\R^2$ constructed using two affine mappings $f_1((x,y)) = \lambda (x, x+y)$ and $f_2((x,y)) = \lambda (x, x+y) + (1,1)$ with some fixed $0<\lambda < \frac{1}{2}$. The fact that $\HH^t(E)= \infty$ follows by observing that $\diam(X_\iii)$ is essentially $|\iii|\lambda^{|\iii|}$ and that there is enough separation among the construction pieces, see \cite[Example 6.4]{KV2009}.
\end{remarks}

The finite clustering property is not always easy to check. Therefore we make the following definition. A WCMC satisfies the \emph{ball condition} if there exists a constant $0 < \delta < 1$ such that for every $x \in E$ we can take a radius $r_x>0$ so that with every $0 < r < r_x$ there is a set $\{x_\iii : \dist(x_\iii,X_\iii) < r, \iii \in Z(x,r)\}$ for which the collection $\{B(x_\iii, \delta r) : \iii \in Z(x,r)\}$ is pairwise disjoint.

We give now a basic example of a WCMC on the Euclidean plane $\R^2$ which, in general, is  not a CMC. 

\begin{example}
We define a self-affine set $E$ using two affine mappings. Choose $0<a_0,a_1,b_0,b_1<1$ so that
$a_0 + a_1 \le 1$ and  $b_0+b_1 \le 1$. Let $c = 1-a_1$ and $d=1-b_1$. We define $f_0,f_1 \colon \R^2 \to \R^2$ by setting
\begin{align*}
  f_0(x,y) &= (a_0  x, b_0 y), \\
  f_1(x,y) &= (a_1 x + c, b_1 y + d)
\end{align*}
for  $x,y \in \R$. The unit square $Q = [0,1] \times [0,1]$ is mapped into itself by these mappings with $f_0(Q) = [0,a_0]\times [0,b_0]$ and $f_1(Q) = [1-a_1,1]\times [1-b_1,1]$. We let $I = \{0,1\}$ and for each $\iii = (i_1, i_2, \dots i_k)$, $k \in \N$, define
\[
X_\iii = f_{i_1}\circ f_{i_2} \circ \dotsc \circ f_{i_k}(Q)
\]
which is a rectangle of width $a_\iii = a_{i_1}a_{i_2}\cdots a_{i_k}$ and height $b_\iii = b_{i_1}b_{i_2}\cdots b_{i_k}$. To see that $\{X_\iii : \iii \in I^*\}$ is a WCMC, note that
\[
  \tfrac 12 (a_\iii + b_\iii) \le \max\{a_\iii, b_\iii\} < \diam(X_\iii) = \sqrt{a_\iii^2 + b_\iii^2} < a_\iii + b_\iii
\]
for $\iii \in I^*$, from which \ref{cond:cover} and \ref{cond:doub} easily follow (for a suitably large $D>1$). Conditions \ref{cond:incl} and \ref{cond:decr} are trivial to check.

Let $x_\iii = (u_\iii,v_\iii)$ be the center point of $X_\iii$ for $\iii \in I^*$.
By looking at the coordinates separately we get for  $\iii \perp \jjj$
\begin{align*}
 d(x_\iii, x_\jjj) \ge & \sqrt{\left(\frac{a_\iii+a_\jjj}{2}\right)^2+\left(\frac{b_\iii+b_\jjj}{2}\right)^2} \ge \frac{a_\iii+a_\jjj+b_\iii+b_\jjj}{4}\\
\ge & \frac{1}{4}\sqrt{a_\iii^2+b_\iii^2}+\frac{1}{4}\sqrt{a_\jjj^2+b_\jjj^2} = \frac{1}{4}\diam(X_\iii)+\frac{1}{4}\diam(X_\jjj).
\end{align*}
Thus for $\iii \perp \jjj$ we have
\[
  B(x_\iii,\frac 14\diam(X_\iii)) \cap B(x_\jjj,\frac 14\diam(X_\jjj)) = \emptyset.
\]
It follows now from \ref{cond:doub} that the ball condition holds.

The following Proposition \ref{prop:equiv} and Remark \ref{rem:sufficientConditions}(ii) will show that in $\R^2$ the ball condition and the finite clustering property are equivalent. We use this fact to determine the Hausdorff dimension of the limit set $E$. The topological pressure is easily calculated to be
\[
  P(t) = \max\{\log(a_0^t + a_1^t), \log(b_0^t + b_1^t) \},
\]
see \cite[(6.1)]{KV2009}. Let $s= P^{-1}(0)$. It is clear that $0 < s \le 1$ and by Proposition \ref{prop:fctoposfinite} we have $\dimh(E) = s$ and $\HH^s(E) > 0$. We also have $\HH^s(E) < \infty$. To see this, note that $E \subset \bigcup_{\iii \in I^n} B(x_\iii, \frac 12 (a_\iii + b_\iii))$ and the diameters of these balls tend to zero as $n \to \infty$. Also note that $a_0^s+a_1^s \le 1$ and $b_0^s+a_1^s \le 1$. Therefore, for each $\varepsilon > 0$ there is an $n\in \N$ such that
\begin{align*}
  \HH_\varepsilon^s(E) &\le \sum_{\iii \in I^n} (a_\iii + b_\iii)^s \le \sum_{\iii \in I^n} a_\iii^s + \sum_{\iii \in I^n} b_\iii^s \\
    &= (a_0^s+a_1^s)^n + (b_0^s + b_1^s)^n \le 2.
\end{align*}
Consequently, $0 < \HH^s(E) \le 2$.
\end{example}

We will show that the ball condition is equivalent to the finite clustering property under some natural conditions for the space or for the WCMC. We start by tracking down how certain bounds for possible cardinalities of collections of disjoint balls with equal radii affect the situation (see Proposition \ref{prop:equiv} below).

Let $A \subset M$ and $r >0$. We call a collection of balls $\{B(x,r) : x \in H\}$ an \emph{$r$-packing} of the set $A$, if $H \subset A$ and $B(x,r)\cap B(y,r) = \emptyset$ for every $y,x \in H$, $x \ne y$. Furthermore we call the packing maximal, if 
\[
 A \subset \bigcup_{y \in H} B(y,2r).
\]
With these notions we can formulate our primary conditions as to when the finite clustering property and the ball condition imply each other. This generalizes \cite[Proposition 3.5]{KV2008}.
\begin{proposition}\label{prop:equiv}
Suppose we have a WCMC. Let $c, r_0 > 0$ and $\alpha_1 \ge \alpha_2 > 0$ be constants. Assume that for every $x \in M$ and $0 < r < R < r_0$, and for every maximal $r$-packing $\{B(x,r) : x \in H\}$ of $B(x,R)$ we have
\begin{equation}\label{eq:ballcond2}
   \#H < c \left(\frac{R}{r}\right)^{\alpha_1}
\end{equation}
Then the ball condition implies the finite clustering property.
If we, on the other hand, have
\begin{equation}\label{eq:ballcond1}
 \#H > c^{-1}\left(\frac{R}{r}\right)^{\alpha_2}
\end{equation}
then the finite clustering property implies the ball condition.
\end{proposition}
\begin{proof}
Assume that \eqref{eq:ballcond2} and the ball condition hold. Take $x \in E$ and let $0 < r < \min\{r_x, 5^{-1}r_0\}$. For every $\iii \in Z(x,r)$ choose a point $x_\iii$ so that the collection $\{B(x_\iii, \delta r) : \iii \in Z(x,r)\}$ is pairwise disjoint. Now
\[
 d(x_\iii , x) \le \dist(x_\iii, X_\iii) + \diam(X_\iii) + r \le 3r
\]
and therefore by \eqref{eq:ballcond2} we have
\[
\#Z(x,r)\le c\left(\frac{4r}{\delta r}\right)^{\alpha_1} = c\left(\frac{4}{\delta}\right)^{\alpha_1}.
\]
Thus the WCMC has the finite clustering property.

Assume now \eqref{eq:ballcond1} and the finite clustering property. Then there exists $L > 0$ such that for every $x \in E$ there is $0 <r_x < r_0$ so that $\#Z(x,r) < L$ whenever $0 < r < r_x$. Define 
\[
 \delta = \frac{1}{2}(Lc)^{-\frac{1}{\alpha_2}}.
\]
For each $\iii \in Z(x,r)$ choose a point $y_\iii \in B(x,r)\cap X_\iii$. We will find the disjoint collection of balls $B(x_\iii, \delta r)$ with centers inside the balls $B(y_\iii,r)$. Let us write $Z(x,r) = \{\iii_j : j = 1, \dots, \#Z(x,r)\}$. Now as the first center, $x_{\iii_1}$, choose any point from $B(y_{\iii_1},r)$. Rest will be chosen by induction. Assume that for $0 < k < \#Z(x,r)$ the points $x_{\iii_j}$, $j = 1, \dots, k$ have been chosen. The claim is that there exists a point 
\[
 x_{\iii_{k+1}} \in B(y_{\iii_{k+1}}, r) \setminus \bigcup_{j = 1}^k B(x_{\iii_j},2\delta r).
\]
Assume the contrary. Now writing down the inequality \eqref{eq:ballcond1} gives
\[
k > c^{-1}\left(\frac{r}{2\delta r}\right)^{\alpha_2} = L,
\]
a contradiction. Ball condition is then satisfied.
\end{proof}

\begin{remarks}\label{rem:sufficientConditions}
(i) We can achieve the equivalence between the ball condition and the finite clustering property by requiring the existence of a measure $\mu$ on $M$ so that for every $x \in M$ and $0 < r < R < r_0$ we have
\begin{equation}\label{eq:measurecond}
 c^{-1}\left(\frac{R}{r}\right)^{\alpha_2} < \frac{\mu(B(x,R))}{\mu(B(x,r))} < c \left(\frac{R}{r}\right)^{\alpha_1}. 
\end{equation}
This forces the inequalities \eqref{eq:ballcond2} and \eqref{eq:ballcond1} to hold: if we let $H$ be as in Proposition \ref{prop:equiv}, then by comparing the measures we get
\[
 \#H c^{-1}\left(\frac{r}{4R}\right)^{\alpha_1} < \sum_{y \in H}\frac{\mu(B(y,r))}{\mu(B(y,4R))} \le \sum_{y\in H}\frac{\mu(B(y,r))}{\mu(B(x,2R))} \le 1
\]
and
\begin{align*}
\frac{1}{c4^{\alpha_1}} &< \frac{\mu(B(x,R))}{\mu(B(x,4R))} \le \sum_{y\in H}\frac{\mu(B(y,2r))}{\mu(B(x,4R))}\\
& \le \sum_{y\in H}\frac{\mu(B(y,2r))}{\mu(B(y,2R))} < \#Hc\left(\frac{r}{R}\right)^{\alpha_2}. 
\end{align*}

(ii) Assume that the space $M$ is Ahlfors $s$-regular, which means that there exists a measure $\mu$ on $M$ and constants $r_0, c>0$ so that
\[
 c^{-1}r^s \le \mu(B(x,r)) \le cr^s
\]
for every $x \in M$ and $0 < r < r_0$. The measure $\mu$ now satisfies the condition \eqref{eq:measurecond}. Hence by Proposition \ref{prop:equiv} and the remark above, the finite clustering property and the ball condition are equivalent. This holds, in particular, in $\R^n$ (which is $n$-regular).

(iii) Now assume that the space $M$ contains at least two points and is \emph{uniformly perfect}, which means that there exists a constant $C > 1$ so that for each $x \in M$ and for each $r > 0$ the set $B(x,r) \setminus B(x, r/C)$ is nonempty whenever the set $M \setminus B(x,r)$ is nonempty. In this situation the inequality \eqref{eq:ballcond1} holds. To see this, let $0 < r_0 < \frac 12 \diam(M)$ and define $\delta = (2C+1)^{-1}$. Then for $0 < R < r_0$ and $x \in M$ the set $M \setminus B(x,R-\delta R)$ is nonempty and therefore $B(x,R-\delta R) \setminus B(x,\frac{R-\delta R}{C})$ is nonempty. Hence every maximal $\delta R$-packing of $B(x,R-\delta R)$ contains at least $2$ balls. Now iterating this we get that inequality \eqref{eq:ballcond1} holds with $c = 2$ and $\alpha_2 = -\frac{\log 2}{\log \delta}$.
Consequently, the finite clustering property implies the ball condition in a uniformly perfect space.

(iv) Finally, let us assume that the space $M$ is doubling. \emph{Doubling} means that there exists a constant $\kappa \in \N$ so that every ball $B(x,2r)$ can be covered with $\kappa$ balls of radius $r$. Now for $0<r<R$ let $n \in N$ be so that $2^{-n}R \le r < 2^{-n+1}R$. Let $H$ be as in Proposition \ref{prop:equiv}. For any point $y \in M$ there can be at most one point in $H \cap B(y,2^{-n-1}R)$. Therefore by iterating the doubling condition we get
\[
 \#H \le K^{n+1} = K^2 (2^{n-1})^{\log_2 K} \le K^2 \left(\frac{R}{r}\right)^{\log_2 K}.
\]
The inequality \eqref{eq:ballcond2} then holds. We conclude that for a WCMC defined on a doubling metric space, the ball conditon implies the finite clustering property.
\end{remarks}

In the remaining part of this section we will work under the assumption that $M$ is doubling. With doubling metric spaces we can make use of certain nicely behaved embeddings of these spaces into Euclidean spaces. 
Working with several spaces and metrics at the same time, we will emphasize the corresponding space, metric or construction with a subscript in the notation whenever there is a possibility of confusion. The standard Euclidean distance function $(x,y) \mapsto|x-y|$ will be denoted by $d_e$. Accordingly, $\diam_e(A)$ will mean the Euclidean diameter of $A$ and $\dist_e(A,B)$ the Euclidean distance between $A$ and $B$
for $A,B \subset \R^n$ (with any $n\in\N$).

From a metric $d$ on $M$ we can derive a \emph{snowflaked metric} for a parameter $0 < p < 1$ by defining $d^p(x,y) = (d(x,y))^p$. A celebrated theorem of P. Assouad \cite[Proposition 2.6]{A1983} gives then the following.

\begin{theorem}\label{thm:assouad}
 Let $(M, d)$ be a doubling metric space. Then for each $0 < p <1$ there exists $n \in \N$ and a bi-Lipschitz embedding 
\[
f \colon (M, d^p) \to (\R^n, d_e).
\]
\end{theorem}

In the next proposition we see that the WCMC structure is preserved under the embedding of Theorem \ref{thm:assouad}. However, it is not clear if all the separation conditions can be transferred in both directions with the embedding. In particular, the ball condition uses points from a neighborhood of the construction pieces and when the ball condition is considered in $\R^n$ these points might lie in  $\R^n \setminus f(M)$.

Proposition 3.7 lists the properties which behave well under the embedding: the Hausdorff measures, topological pressure and finite clustering property can be transferred back and forth between the spaces, whereas the ball condition and tractability can be pushed to the image side. This last property is defined as follows: a WCMC is \emph{tractable} if there is a constant $C \ge 1$ such that for each $r > 0$ we have
\[
 \dist(X_{\hhh\iii}, X_{\hhh\jjj}) \le C \diam(X_\hhh)r
\]
whenever $\hhh \in I^*$, $\iii, \jjj \in Z(r)$, and $\dist(X_\iii, X_\jjj) \le r$.

\begin{proposition}\label{prop:AssouadPushing}
Let $\MM = \{X_\iii : \iii \in I^*\}$ be a WCMC (or CMC) in a doubling metric space $M$ and $p$, $n$ and $f$ as in Theorem \ref{thm:assouad}. Then
\begin{enumerate}
 \item $\MM' := \{f(X_\iii) : \iii \in I^*\}$ is a WCMC (or CMC respectively) in $\R^n$,
 \item $P_\MM(pt) = P_{\MM'}(t)$ for every $t \ge 0$,
 \item there exists a constant $C>0$ so that 
       \[C^{-1}\HH_e^s(f(A)) \le \HH_d^{sp}(A) \le C\HH_e^s(f(A))\]
       for every Borel set $A \subset M$,
 \item The following three are equivalent:
\begin{enumerate}
\item $\MM$ has the finite clustering property,
\item $\MM'$ has the finite clustering property,
\item $\MM'$ satisfies the ball condition.
\end{enumerate}
These three conditions also hold if $\MM$ satisfies the ball condition,
 \item if  $\MM$ is tractable, then $\MM'$ is tractable.
\end{enumerate}
\end{proposition}

\begin{proof}
We will prove the proposition for a WCMC. The proof for a CMC is similar. Let $L$ be the bi-Lipschitz constant of $f$ and constants $c$ and $\roo$ from Lemma \ref{lem:diamBoundExponential}. Assume $\MM$ is a WCMC. Take $t \ge 0$. Since
\[
 L^{-1} \diam_e(f(X_\iii))^t \le \diam_d(X_\iii)^{pt} \le L \diam_e(f(X_\iii))^t,
\]
(2) and (3) are true.

Let us check (1). For $\MM'$ the condition \ref{cond:incl} is obvious. To see \ref{cond:cover} we calculate for every $\iii,  \jjj \in I^*$
\begin{align*}
 \diam_e(f(X_{\iii\jjj})) &\le L \diam_d(X_{\iii\jjj})^p \le LD^p \diam_d(X_\iii)^p\diam_d(X_\jjj)^p \\
& \le L^3D^p \diam_e(f(X_{\iii})) \diam_e(f(X_{\jjj})).
\end{align*}
Similarly for \ref{cond:doub} we get
\begin{align*}
 \diam_e(f(X_{\iii})) &\ge L^{-1} \diam_d(X_{\iii})^p \ge L^{-1}D^{-p} \diam_d(X_{\iii^-})^p \\
& \ge L^{-2}D^{-p}\diam_e(f(X_{\iii^-})).
\end{align*}
Finally \ref{cond:decr} follows from \eqref{eq:geomdecr} with large enough $n \in \N$ by
\[
 \max_{\iii \in I^n}\diam_e(f(X_{\iii})) \le L \max_{\iii \in I^n}\diam_d(X_{\iii})^p \le cL\roo^n < D^{-1}.
\]
We denote by $D'$ the constant $D$ for $\MM'$ in the definition of a WCMC.

Next we prove (4). Since the ball condition implies the finite clustering property in a doubling metric space, we only need to prove the three equivalences. Assume that $\MM$ has the finite clustering property. Take $x \in E$ and $r >0$. First we notice that
\[
 f^{-1}(B_e(f(x),r)) \subset B_d(x, (Lr)^{1/p}).
\]
Take $\iii \in Z_\MM(x, (Lr)^{1/p})$. Now $\diam_e(f(X_\iii)) \le L^2r$. Let $l$ be the smallest integer which satisfies
\[
 l > \frac{-\log(D'L^3c)}{\log(\roo^p)}.
\]
Now for any $\hhh \in I^l$
\[
 \diam_e(f(X_{\iii\hhh})) \le D'\diam_e(f(X_\iii))\diam_e(f(X_\hhh)) \le D'\cdot L^2r \cdot Lc\roo^{pl} \le r.
\]
Therefore
\[
 \#Z_{\MM'}(f(x),r) \le N^l \# Z_\MM(x, (Lr)^{1/p})
\]
so the finite clustering property holds for $\MM'$. 

Assume that $\MM$ does not have the finite clustering property. Take $M \in \N$. There exists a point $x \in X$ such that $\limsup_{r \downarrow 0} \#Z_\MM(x,r) > M$. Let $r>0$ be small so that $\#Z_\MM(x,r) \ge M$. Fix $m \in \N$ so that $m > \frac{\log(L^{2/p}Dc)}{-\log \roo}$. Our claim is that 
\begin{equation}\label{eq:RealClustering}
\#Z_{\MM'}(f(x),Lr^p) \ge \frac{M}{N^m}.
\end{equation}

Take any $\iii \in Z_\MM(x,r)$ and for it find $k \in \N$ so that
\[
 \diam_e(f(X_{\iii|_k})) \le Lr^p < \diam_e(f(X_{\iii|_{k-1}})).
\]
Now $f(X_{\iii|_k}) \cap B_e(x,Lr^p) \ne \emptyset$ and by \ref{cond:decr} and \eqref{eq:geomdecr} we have
\[
 \frac{r^p}{L} < \diam_e(f(X_\iii))\le LD^p\diam_d(X_{\iii|_k})^p\diam_d(X_\jjj)^p \le LD^pr^pc^r\roo^{p|\jjj|},
\]
where $\iii = \iii|_k\jjj$. From the choice of $m$ we see that $|\jjj| \le m$ and thus \eqref{eq:RealClustering} holds. Therefore $\MM'$ does not have the finite clustering property.

Lastly, because $\MM'$ is a WCMC in $\R^n$, the finite clustering property and the ball condition are equivalent.


We are left with proving (5). Take $r>0$, $\hhh \in I^*$ and $\iii, \jjj \in Z_{\MM'}(r)$ so that $\dist_e(f(X_\iii), f(X_\jjj)) \le r$. Now $\dist_d(X_\iii, X_\jjj) \le (Lr)^{1/p}$ and 
\[
\max\{\diam_d(X_\iii), \diam_d(X_\jjj)\} \le (Lr)^{1/p}.
\]
Let $\iii', \jjj' \in Z_\MM((Lr)^{1/p})$ so that $[\iii] \subset [\iii']$ and $[\jjj] \subset [\jjj']$. Because $\MM$ is tractable, we have
\[
 \dist_d(X_{\hhh\iii'}, X_{\hhh\jjj'}) \le C\diam_d(X_\hhh)(Lr)^{1/p}.
\]
Therefore
\begin{align*}
\dist_e(f(X_{\hhh\iii'}),f(X_{\hhh\jjj'}))&  \le L\dist_d(X_{\hhh\iii'}, X_{\hhh\jjj'})^p \le L^2C^p\diam_d(X_\hhh)^pr\\
&\le L^3C^p\diam_e(f(X_\hhh))r.
\end{align*}
On the other hand we get
\begin{align*}
 \diam_e(f(X_{\hhh\iii'})) &\le D'\diam_e(f(X_\hhh))\diam_e(f(X_{\iii'}))\\
& \le D'\diam_e(f(X_\hhh)) L\diam_d(X_{\iii'})^p  \\
& \le L^2D'\diam_e(f(X_\hhh))r
\end{align*}
and the same estimate for $\diam_e(f(X_{\hhh\jjj'}))$. By combining these estimates we get
\begin{align*}
 \dist_e(f(X_{\hhh\iii}),f(X_{\hhh\jjj})) \le & \diam_e(f(X_{\hhh\iii'})) +\diam_e(f(X_{\hhh\jjj'}))\\
& +\dist_e(f(X_{\hhh\iii'}),f(X_{\hhh\jjj'}))\\
 \le & (L^3C^p + 2L^2D')\diam_e(f(X_\hhh))r
\end{align*}
and we are done.
\end{proof}

As a first consequence of Proposition \ref{prop:AssouadPushing} we prove the following result.

\begin{proposition}\label{prop:tractCMC}
 A tractable CMC in a doubling metric space has the finite clustering property if $\HH^t(E) > 0$ with $t = P^{-1}(0)$.
\end{proposition}
\begin{proof}
 This result is true in $\R^n$, \cite[Theorem 3.9]{KV2008}. Let $\{X_\iii : \iii \in I^*\}$ be a tractable CMC in $M$ so that $\HH^t(E) > 0$ with $P(t)=0$. Take $0 < p < 1$. Then by Proposition \ref{prop:AssouadPushing} $\{f(X_\iii) : \iii \in I^*\}$ is a tractable CMC in $\R^n$ with $P(\frac{t}{p}) = 0$ and $\HH_e^{t/p}(f(E)) > 0$. Therefore we know that $\{f(X_\iii) : \iii \in I^*\}$ satisfies the finite clustering property. The finite clustering property for the original CMC follows then from Proposition \ref{prop:AssouadPushing}.
\end{proof}

\section{Semiconformal iterated function systems}\label{section:semiconformal}

Assume that $M$ is a complete metric space and that for each $i \in I$ there is a contractive injection $\varphi_i \colon M \to M$. By contractivity of a mapping $\varphi$ we mean that there is a constant $0 < s < 1$ so that
\[
 d(\varphi(x),\varphi(y)) \le s d(x,y)
\]
for every $x,y \in M$. The collection $\{\varphi_i : i \in I\}$ is called an \emph{iterated function system (IFS)}. As is well known, there is a unique nonempty compact set $E \subset M$ (which we call the \emph{invariant set} of the IFS) such that
\[
 E = \bigcup_{i \in I}\varphi_i(E).
\]
We call the contractive mapping $\varphi_i$ a \emph{similitude} if there exists a fixed ratio
$0 < r_i < 1$ such that $d(\varphi_i(x),\varphi_i(y)) = r_i d(x,y)$ for every $x,y \in M$.
If all the mappings of the IFS are similitudes, the invariant set is called \emph{self-similar}.

Write $\varphi_\iii = \varphi_{i_1} \circ \cdots \circ \varphi_{i_n}$ for $\iii = (i_1, \dots, i_n)$ and $n \in \N$. We say that the IFS is \emph{semiconformal} if the invariant set $E$ has positive diameter and there are constants $D \ge 1$ and $0 < \underline{s}_\iii \le \overline{s}_\iii < 1$ (for each $\iii \in I^*$) such that $\overline{s}_\iii \le D \underline{s}_\iii$ and
\begin{equation}\label{eq:semiconformality}
 \underline{s}_\iii d(x,y) \le d(\varphi_\iii(x), \varphi_\iii(y)) \le \overline{s}_\iii d(x,y)
\end{equation}
for any $x,y \in M$ and $\iii \in I^*$. Note that then
\begin{equation}\label{eq:semiequal}
 \frac{D^{-1}}{\diam(E)} \diam(\varphi_\iii(E)) \le \underline{s}_\iii \le \overline{s}_\iii \le \frac{D}{\diam(E)} \diam(\varphi_\iii(E))
\end{equation}
for each $\iii \in I^*$.


The following was proved in \cite[Lemma 5.1, Lemma 5.2]{KV2008} for semiconformal IFSs in $\R^d$. Although the proof is the same in metric spaces, we repeat it here.

\begin{proposition}\label{prop:semiIFStractable}
Let $E$ be the invariant set of a semiconformal IFS $\{\varphi_i : i \in I\}$. Then $\{\varphi_\iii(E) : \iii \in I^*\}$ is a tractable CMC. 
\end{proposition}
\begin{proof}
 Let us first prove that $\{\varphi_\iii(E) : \iii \in I^*\}$ is a CMC. By semiconformality, we have $\diam(E) > 0$. Since $E = \bigcup_{i\in I} \varphi_i(E)$, \ref{cond:incl} is satisfied. From \eqref{eq:semiconformality} and \eqref{eq:semiequal} we get
\begin{align*}
 \diam(\varphi_{\iii\jjj}(E)) &\le \overline{s}_\iii \diam(\varphi_{\jjj}(E)) 
\\ &\le \frac{D}{\diam(E)}\diam(\varphi_{\iii}(E))\diam(\varphi_{\jjj}(E))
\end{align*}
and
\begin{align*}
\diam(\varphi_{\iii\jjj}(E)) &\ge \underline{s}_\iii \diam(\varphi_{\jjj}(E)) 
\\ & \ge \frac{D^{-1}}{\diam(E)}\diam(\varphi_{\iii}(E))\diam(\varphi_{\jjj}(E)).
\end{align*}
Thus \ref{cond:M2} holds. Contractivity of $\varphi_i$ for every $i \in I$ ensures \ref{cond:decr}.

To see tractability, assume $r>0$, take any $\hhh \in I^*$ and choose $\iii, \jjj \in Z(r)$ so that $\dist(\varphi_\iii(E), \varphi_\jjj(E)) \le r$. Then
\begin{align*}
 \dist(\varphi_{\hhh\iii}(E), \varphi_{\hhh\jjj}(E)) &\le \overline{s}_\hhh \dist(\varphi_{\iii}(E), \varphi_{\jjj}(E)) \\
     &\le D\underline{s}_\hhh \dist(\varphi_{\iii}(E), \varphi_{\jjj}(E)) \\
     &\le \frac{D}{\diam(E)} \diam(\varphi_\hhh(E))\, r,
\end{align*}
and we are done.
\end{proof}

In the sequel we will denote $\varphi_\iii(E)$ by $E_\iii$ whenever the need to simplify the notation arises. 
Given an IFS  $\{\varphi_i\}_{i\in I}$, the set system $\{ E_\iii \}_{\iii \in I^*}$ is not necessarily a WCMC but when it is, we call the topological pressure of the WCMC also the pressure of the corresponding IFS. By \eqref{eq:semiequal} it is clear that the pressure of a semiconformal IFS can be calculated by the formula
\[
  P(t) = \lim_{n \to \infty} \frac{1}{n}\log\sum_{\iii \in I^n} s_\iii^{t}
\]
where each $s_\iii$, $\iii \in I^*$, is allowed to be any of the numbers $\diam(E_\iii)$, $\underline{s}_\iii$ or $\overline{s}_\iii$.
 In the special case that every $\varphi_i$ is a similitude and $r_i$, $i\in I$, are the corresponding contraction ratios, the most natural choice for $s_\iii$ indexed by $\iii = (i_1,\dots,i_n) \in I^{|\iii|}$ is $s_\iii = r_{i_1}\cdots r_{i_n}$. Then the equation $P(t) = 0$ simplifies to the so-called \emph{Moran equation}
\[
   \sum_{i \in I} r_i^t = 0.
\]
The solution of this equation is usually called the \emph{similarity dimension} of the corresponding similitude IFS.

We say that an IFS satisfies the ball condition if the iterated images of the invariant set constitute a WCMC that satisfies the ball condition. We define the finite clustering property for an IFS similarly. The next proposition and its corollary show that if the IFS in question is semiconformal and defined on a doubling space, then the ball condition is in fact equivalent to the finite clustering property.

\begin{proposition}\label{prop:wegotballs}
 Let $M$ be a complete doubling metric space and $\{\varphi_i\}_{i \in I}$ a semiconformal IFS on $M$ such that 
$\{ \varphi_\iii(E) \}_{\iii \in I^*}$ has the finite clustering property. Then there is a constant $\delta > 0$ and a point $x \in E$ so that
\[
 B(\varphi_\iii(x), \delta \diam(E_\iii)) \cap B(\varphi_\jjj(x), \delta \diam(E_\jjj)) = \emptyset
\]
whenever $\iii \perp \jjj$.
\end{proposition}
\begin{proof}
Let $0 < p < 1$. With the Assouad embedding $f\colon (M,d^p) \to (\R^n,d_e)$ we get a tractable CMC $\{f(E_\iii)\}_{\iii \in I^*}$ on $\R^n$. By Proposition \ref{prop:AssouadPushing} it satisfies the ball condition. Furthermore, letting $L$ denote the bi-Lipschitz constant of $f$, it is straightforward to check, simply by using the definitions, that by choosing $C^* = L^4 D^{2p}$ (where the constant $D \ge 1$ is from the definition of semiconformality) we get
\[
  \frac{\dist_e(f(E_{\hhh\iii}), f(E_{\hhh\jjj}))}{\diam_e(f(E_\hhh))} \le C^* \frac{\dist_e(f(E_{\kkk\iii}), f(E_{\kkk\jjj}))}{\diam_e(f(E_\kkk))}
\]
for all $\iii,\jjj,\hhh,\kkk \in I^*$. Thus $\{f(E_\iii)\}_{\iii \in I^*}$ is, using the terminology of \cite{KV2008}, a semiconformal CMC. This property allows us to utilize \cite[Corollary 4.8]{KV2008} to get a constant $\delta' > 0$ and a point $x \in E$ so that
\[
 B(f(\varphi_\iii(x)), \delta' \diam_e(f(E_\iii))) \cap B(f(\varphi_\jjj(x)), \delta' \diam_e(f(E_\jjj))) = \emptyset
\]
whenever $\iii \perp \jjj$. Now by combining the facts that
\[
  f^{-1}(B(f(z),r)) \supset B(z,(L^{-1} r)^{1/p})
\]
for $z \in M$, $r>0$ and $\diam_e(f(E_\iii)) \ge L^{-1} \diam(E_\iii)^{p}$ for $\iii \in I^*$, we find that with $\delta = (\delta'L^{-2})^{1/p}$ we have
\[
 B(\varphi_\iii(x), \delta \diam(E_\iii)) \cap B(\varphi_\jjj(x), \delta \diam(E_\jjj)) = \emptyset
\]
whenever $\iii \perp \jjj$.
\end{proof}

\begin{corollary}\label{cor:BallCondEquivPressureCond}
For a CMC $\{ \varphi_\iii(E) \}_{\iii \in I^*}$ corresponding to a semiconformal IFS defined on a complete doubling metric space, the following conditions are equivalent:
\begin{enumerate}
 \item The ball condition.
 \item The finite clustering property.
 \item $\HH^t(E) > 0$ with $P(t) = 0$.
 \item There exist $x\in M$ and $\eps > 0$ such that
\begin{equation}\label{eq:properseparation}
  d(\varphi_\iii(x),\varphi_\jjj(x)) \ge \eps(\underline{s}_\iii + \underline{s}_\jjj) \,\text{ whenever $\iii \perp \jjj$.}
\end{equation}
\end{enumerate}
\end{corollary}
\begin{proof}
  Assume that $\{ \varphi_i \colon M \to M \}_{i \in I}$ is a semiconformal IFS and $M$ is doubling. If the ball condition is satisfied, then by Remark \ref{rem:sufficientConditions}(iv) the corresponding CMC has the finite clustering property.
Proposition \ref{prop:wegotballs} gives the other direction. By Proposition \ref{prop:semiIFStractable}, the corresponding CMC is tractable. Hence Propositions \ref{prop:fctoposfinite} and \ref{prop:tractCMC} together give the equivalence
between the finite clustering property and the positivity of $\HH^t(E)$ for $t=P^{-1}(0)$. Consequently, the first three conditions are equivalent. Furthermore, it follows immediately from \eqref{eq:semiequal} and Proposition \ref{prop:wegotballs} that the ball condition implies the fourth condition. 

As the final step, we will show that the last condition implies the ball condition. Assume that $x\in M$ and $\eps>0$ satisfy \eqref{eq:properseparation}. Let $0<r<\diam(E)$. By \ref{cond:doub} and the definition of $Z(r)$, there is a constant $d_0 > 0$, not depending on $r$, such that $d_0 r \le \diam(E_\iii) \le r$ for all $\iii \in Z(r)$. Now choose $\hhh\in I^*$ long enough so that
\[
  \frac{D}{\diam(E)}\, d(x,E)\, \overline{s}_\hhh\le 1
\]
and take $x_\iii = \varphi_{\iii\hhh}(x)$ for each $\iii \in I^*$. Then, using \eqref{eq:semiequal}, for each $\iii \in Z(r)$ we have
\begin{align*}
  d(x_\iii,E_\iii) &\le \overline{s}_\iii d(\varphi_\hhh(x),E) \le \frac{D}{\diam(E)}\diam(E_\iii)\, d(\varphi_\hhh(x),E_\hhh) \\
&\le \frac{D}{\diam(E)}\,r\, \overline{s}_\hhh d(x,E) \le r.
\end{align*}
Moreover, using \eqref{eq:semiequal} this time twice, we get
\begin{align*}
  d(x_\iii,x_\jjj) &\ge \eps \,(D\diam(E))^{-2}\diam(E_\hhh) (\diam(E_\iii) + \diam(E_\jjj))\\
  &\ge 2\eps \,(D\diam(E))^{-2}\diam(E_\hhh)\,d_0 r
\end{align*}
for distict $\iii,\jjj \in Z(r)$. This implies that by choosing
\[
\delta = \eps \,(D\diam(E))^{-2}\diam(E_\hhh)\,d_0
\]
we get $B(x_\iii,\delta r) \cap B(x_\jjj,\delta r) = \emptyset$ for any two distinct $\iii,\jjj \in Z(z,r)$ with any $z\in M$. Thus the ball condition holds, and the proof is complete.
\end{proof}

Our next effort is to relate the ball condition to a more familiar separation condition defined here as follows. An IFS satisfies the \emph{open set condition (OSC)} if there exists a nonempty open set $U \subset M$ such that
\begin{equation}\label{eq:OSC}
 \varphi_\iii(U) \cap \varphi_\jjj(U) = \emptyset \qquad \text{whenever } \iii \perp \jjj.
\end{equation}
We call such an open set $U$ \emph{feasible} for the OSC. If there exists a feasible $U$ for which $U \cap E \ne \emptyset$, the IFS satisfies the \emph{strong open set condition (SOSC)}.

\begin{remark}\label{rem:standardOSC}
The standard version of the OSC, from \cite{H1981}, assumes the existence of a nonempty open set $\mathcal{O} \subset M$ such that $\varphi_i(\mathcal{O}) \subset \mathcal{O}$ for each $i \in I$ and $\varphi_i(\mathcal{O}) \cap \varphi_j(\mathcal{O}) = \emptyset$ whenever $i,j \in I$ and $i \ne j$. By assuming further that $\mathcal{O}$ intersects the invariant set $E$, we get the standard SOSC. As regards to when the standard versions of the OSC and SOSC are equivalent to our versions, this certainly holds if the nonempty open set $U$ satisfying \eqref{eq:OSC} can be chosen so that the set $\mathcal{O} := U \cup \bigcup_{\iii \in I^*} \varphi_\iii(U)$ is open as well, because then $\mathcal{O}$ is a feasible open set for the standard OSC and it intersects $E$ if $U$ intersects $E$. We refer to the proof of \cite[lemma 5.3]{KV2008} for details.
%
\end{remark}

The ball condition implies the OSC for a semiconformal IFS.
We defer the easy verification of this fact until later (see the proof of Theorem \ref{thm:equiv}).
Instead, we show now by a simple example that the reverse implication is not generally true, not even for a similitude IFS defined on a complete doubling metric space. The example also shows that the OSC and the SOSC are not equivalent in the setting of doubling metric spaces.

\begin{example}\label{ex:OSCnoSOSC}
Let $\frac 12 < r < 1$ and consider the pair $\varphi_0,\varphi_1$ of similitudes defined at each $x \in \R^2$ by
\begin{equation}\label{eq:overlappingmappings}
  \varphi_0(x) = r x, \quad \varphi_1(x) = r x + (1,0).
\end{equation}
Letting $\mathcal{I} = \{(x,0) \in \R^2 : 0 \le x \le \frac{1}{1-r} \}$, it is easy to check that
\[
  \mathcal{I} = \varphi_0(\mathcal{I}) \cup \varphi_1(\mathcal{I}).
\]
This means that the horizontal line segment $\mathcal{I}$ is the invariant set of the similitude IFS $\{\varphi_0,\varphi_1 \}$. The similarity dimension of this IFS, denoted here by $s$, satisfies the Moran equation
$
 r^s + r^s = 1
$
so we have
\begin{equation}\label{eq:DimTooMuch}
 s = P^{-1}(0) = \frac{\log 2}{\log(1/r)} > 1 = \dimh(\mathcal{I}).
\end{equation}
Now let $\mathcal{J} = \{(0,y) \in \R^2 : 0 \le y \le 1 \}$ and set
\[
  M = \,\mathcal{I} \cup \mathcal{J} \cup \bigcup_{\iii \in \{0,1\}^*} \varphi_\iii(\mathcal{J})\,.
\]
Note that for each $\iii \in \{0,1\}^*$, the set $\varphi_\iii(\mathcal{J})$ is a vertical line segment with lower endpoint $\varphi_\iii(0,0) \in \mathcal{I}$ and of height $r^{|\iii|}$. 
It is simple to check that the complement of $M$ is open, so $M$ itself is closed. Hence, equipped with the inherited Euclidean metric, $M$ is a complete doubling metric space.

Since we have $\varphi_i(M) \subset M$ for $i\in \{0,1\}$, 
we may regard $\{\varphi_0,\varphi_1 \}$ as a similitude IFS on $M$. Due to the strict inequality in \eqref{eq:DimTooMuch}, the SOSC cannot hold for this IFS because under the SOSC, the similarity dimension of the IFS equals the Hausdorff dimension of the invariant set (see \cite[Theorem 2.6]{S1996} or Proposition \ref{prop:SOSCdim} later in this section). For the same reason, recalling Corollary \ref{cor:BallCondEquivPressureCond}, neither is the ball condition satisfied. However, 
the OSC is satisfied for all but countably many values of $r$ in \eqref{eq:overlappingmappings}: letting
\[
  U = \mathcal{J} \setminus \{(0,0)\}
\]
which is open in $M$, we claim that $\varphi_\iii(U) \cap \varphi_\jjj(U) = \emptyset$  whenever $\iii \perp \jjj$ provided that $r$ is a transcendental number. To see this, first note that $\varphi_\iii(U)$ and $\varphi_\jjj(U)$ intersect if and only if $\varphi_\iii(0,0) = \varphi_\jjj(0,0)$. Moreover, given $m \in \N$ and $\iii = (i_1,i_2, \dots i_m) \in \{0,1\}^m$, it is easy to verify by induction that $\varphi_\iii(0,0) = (x_\iii,  0)$ where
\[
 x_{\iii} = \sum_{k=1}^{m} i_k r^{k-1}.
\]
In particular, $\varphi_\iii(0,0) = \varphi_{\iii0}(0,0)$ for each $\iii \in \{0,1\}^*$. Thus if there exist symbols $\iii = (i_1,i_2,\dots, i_m)$ and $\jjj = (j_1,j_2,\dots, j_n)$ in $\{0,1\}^*$ such that $\iii \perp \jjj$ and $\varphi_\iii(U) \cap \varphi_\jjj(U) \ne \emptyset$, there is no loss of generality to assume $m=n$ (extend $\iii$ or $\jjj$ with trailing zeros if necessary). Then
\[
  x_{\iii} - x_{\jjj} = \sum_{k=1}^{m} (i_k - j_k) r^{k-1} = 0
\]
and $i_k - j_k \ne 0$ for at least one $k \in \{2,\dots,m\}$. This all shows that if $U$ is not feasible for the OSC, then $r$ has to be an algebraic number. We conclude that the OSC is satisfied for each transcendental value of $r$. Recall that the set of algebraic numbers is only countable.

Note that if $r$ is transcendental, then by Remark \ref{rem:standardOSC} the standard OSC is satisfied with $\mathcal{O} = U \cup \bigcup_{\iii \in \{0,1\}^*} \varphi_\iii(U) = M\setminus \mathcal{I}$ as the feasible open set.
\end{example}


We now strive for a better situation with respect to separation between disjoint images of a feasible open set than what was observed in the example above. It is in fact easy to see that if there is a feasible open set $U$ such that for every $\iii \in I^*$, one can find  a large enough ball inside $\varphi_\iii(U)$, with radius comparable to the diameter of $\varphi_\iii(U)$, then the ball condition holds. 
Fortunately, in the semiconformal setting there is a natural condition under which every bounded feasible open set $U$ is like this. 
To introduce the condition, we assume that $\mathcal{F} = \{ \varphi_i \colon M \to M \}_{i \in I}$ is a semiconformal IFS with invariant set $E$, and refer any dense open set $W \subset M$ satisfying $W \cap E \ne \emptyset$ as an \emph{essential open set} (for $\mathcal{F}$). We say that $\mathcal{F}$ is \emph{properly semiconformal} if there is an essential open set $W \ne M$ such that for each $x \in W$ there is a constant $\lambda_x \ge 1$ so that
\begin{equation}\label{eq:properdist}
  \dist(\varphi_\iii(x), \varphi_\iii(M \setminus W)) \le \lambda_x \dist(\varphi_\iii(x), M \setminus \varphi_\iii(W) )
\end{equation}
for every $\iii \in I^*$.
The next proposition will put this definition in a proper perspective. Note that with $\iii \in I^*$ and $\overline{s}_\iii$ from \eqref{eq:semiconformality} we always have
\begin{equation}\label{eq:alwaysright}
  \varphi_\iii(B(x,r)) \subset B(\varphi_\iii(x),\overline{s}_\iii r)
\end{equation}
for $x\in M$ and $r>0$, whether $\mathcal{F}$ is properly semiconformal or not.

\begin{proposition}\label{prop:propersemiconf}
   A semiconformal IFS is properly semiconformal if and only if there is an essential open set $W \varsubsetneq M$ such that
for each $x \in W$ there is $r_x > 0$ so that
\begin{equation}\label{eq:propersemiconf}
 B(\varphi_\iii(x),\underline{s}_\iii r) \subset \varphi_\iii(B(x,r)) 
\end{equation}
whenever $x \in W$, $0<r\le r_x$ and $\iii \in I^*$. 
\end{proposition}
\begin{proof}
  First assume that $\{ \varphi_i \}_{i \in I}$ is a properly semiconformal IFS with an essential open set $W$ having the required properties. Take any $x\in W$ with $\lambda_x>0$ as in \eqref{eq:properdist} and choose $R > 0$ so that $B(x,R) \subset W$. We begin by showing that
\[
  B(\varphi_\iii(x), \underline{s}_\iii \lambda_x^{-1}R) \subset \varphi_\iii(W)
\]
for any $\iii \in I^*$.
Assume that for some $\iii \in I^*$ the contrary holds. Then there are points $y \in M\setminus \varphi_\iii(W)$ and $x' \in M\setminus W$ such that $d(\varphi_\iii(x),y) < \underline{s}_\iii \lambda_x^{-1}R$ and $d(\varphi_\iii(x), \varphi_\iii(x')) < \underline{s}_\iii R$. Noticing that $x'\notin B(x,R)$ however leads to the contradiction
\[
    R \le d(x,x') \le \underline{s}_\iii^{-1} d(\varphi_\iii(x), \varphi_\iii(x'))  < R.
\]
Thus for each $y \in B(\varphi_\iii(x), \underline{s}_\iii \lambda_x^{-1}R)$ there is an $x' \in W$ for which $\varphi_\iii(x') = y$. Now choose $r_x = \lambda_x^{-1}R$ and assume that $\iii \in I^*$ and $0 < r \le r_x$. Then with any $y = \varphi_\iii(x') \in B(\varphi_\iii(x), \underline{s}_\iii r)$ we have
\[
  d(x,x') \le \underline{s}_\iii^{-1} d(\varphi_\iii(x),\varphi_\iii(x')) < \underline{s}_\iii^{-1}\cdot\underline{s}_\iii r = r
\]
so that $y \in \varphi_\iii(B(x,r))$. Consequently, we have \eqref{eq:propersemiconf}.

For the reverse implication, take $x\in W$, $\iii \in I^*$ and assume that \eqref{eq:propersemiconf} holds for $0 < r \le r_x$.
Also fix an arbitrary $x_0 \in M\setminus W$. Then we have
$
  B(\varphi_\iii(x), \underline{s}_\iii r_x) \subset \varphi_\iii(W)
$
and thus
\[
  \dist(\varphi_\iii(x), M\setminus\varphi_\iii(W)) \ge \underline{s}_\iii r_x.
\]
On the other hand, with $D \ge 1$ from the definition of semiconformality we get
\[
  \dist(\varphi_\iii(x), \varphi_\iii(M \setminus W)) \le d(\varphi_\iii(x),\varphi_\iii(x_0)) \le D\underline{s}_\iii d(x,x_0).
\]
As a conclusion,
\[
  \dist(\varphi_\iii(x), \varphi_\iii(M \setminus W)) \le Dr_x^{-1} d(x,x_0) \dist(\varphi_\iii(x), M\setminus\varphi_\iii(W))
\]
and we are done.
\end{proof}

\begin{remarks}\label{rmk:semiconf}
(i) Assume that we have a properly semiconformal IFS which satisfies the (S)OSC. Let $E$ be the invariant set and let $U$ be a feasible open set. The denseness of the essential open set $W$ and having $W\cap E \ne \emptyset$ allow us to assume that $U \subset W$. Then Proposition \ref{prop:propersemiconf} clearly implies that $\varphi_\iii(U)$ is open for each $\iii \in I^*$. Therefore, recalling Remark \ref{rem:standardOSC}, the OSC and the SOSC are equivalent to their standard versions. Furthermore, given any feasible open set $U$, we can take $x \in U \cap W$ and choose $0<r<r_x$ such that $B(x,r) \subset U$, and then it is easy to see that \eqref{eq:properseparation} holds with $\eps = \frac r2$. This allows us to conclude that if we have a properly semiconformal IFS, then the OSC implies the ball condition.

(ii) 
Given two IFSs $\{ \varphi_i \colon M \to M \}_{i \in I}$ and $\{ \psi_i \colon M' \to M' \}_{i \in I}$ which are topologically conjugated by a bi-Lipschitz homeomorphism $h\colon M\to M'$ (so that $\psi_i = h\circ\varphi_i\circ h^{-1}$ for each $i \in I$), it is simple to verify that if either one is properly semiconformal then the same holds for the other. In this sense, proper semiconformality is a metric invariant.

(iii) Any semiconformal IFS for which the defining mappings $\varphi_i$ are bijections is properly semiconformal (in the definition choose $W=M\setminus\{x_0\}$ with an arbitrary $x_0 \in M$). Bijectivity was assumed by A. Schief in \cite{S1996} where he studied the self-similar case in complete metric spaces. It was also assumed (although not mentioned in the paper) 
by Z. Balogh and H. Rohner in \cite{BR2007} where they carried on the study of the self-similar case. However, the bijectivity assumption is too strong already in the important special case of conformal iterated function systems on Euclidean spaces, as there are no bijective conformal contractions on $\R^n$ with $n \ge 2$ other than the contractive similitudes. Conformal IFSs and separation conditions for them have been studied extensively. For recent developments,
see \cite{LNW2009} and the references therein.
\end{remarks}

Let us now consider a setting suitable, in particular, for conformal iterated function systems on Euclidean spaces. Assuming here that $M \subset \R^n$, we say that an IFS $\mathcal{F}$ formed by mappings $\varphi_i\colon M \to M$, $i \in I$, is \emph{properly Euclidean} if the Euclidean metric is used and $M$ is the closure, in $\R^n$, of an open set $W \varsubsetneq \R^n$ such that $\varphi_i(M) \subset W$ for each $i \in I$. Then $W$ is an essential open set for $\mathcal{F}$. Another crucial observation is that if $U$ is an open proper subset of $\R^n$ and $x\in U$, then there is a point $z\in \R^n \setminus U$ at minimum distance to $x$, and $z$ is \emph{a fortiori} a boundary point of $U$ (simply because $z + t(x-z) \in U$ for all $0<t\le1$).
Thus, noting that for each $\iii \in I^*$ the closed set $\varphi_\iii(M)$ contains the boundary of the open set $\varphi_\iii(W)$, we have
\[
  \dist(\varphi_\iii(x), M \setminus \varphi_\iii(W)) = \dist(\varphi_\iii(x), \varphi_\iii(M) \setminus \varphi_\iii(W))
\]
for $x\in W$ and $\iii \in I^*$. So \eqref{eq:properdist} holds here with $\lambda_x = 1$. Consequently, any semiconformal IFS which is properly Euclidean is properly semiconformal.

Using similar reasoning, we get the following generalization beyond the Euclidean case: if $M$ is the closure of an open and proper subset $W$ of a complete quasiconvex space and $W$ meets the same criteria as above, then a semiconformal IFS defined on $M$ is always properly semiconformal. Here by a quasiconvex space we mean a metric space $(X,d)$ for which there is a constant $C \ge 1$ such that any two points $x,y\in X$ can be joined by a rectifiable curve of length at most $C d(x,y)$.

\begin{example}\label{ex:ultra}
To get a further example of a situation where semiconformality implies proper semiconformality, this time in a totally disconnected space, assume that for each $i \in I$ there is a contractive mapping $\varphi_i\colon I^\infty \to I^\infty$ on the symbol space $(I^\infty,d_2)$ such that $\varphi_i(C)$ is a cylinder whenever $C$ is a cylinder. Then $\varphi_\iii(I^\infty)$ is a cylinder for each $\iii \in I^*$.
Choose an arbitrary $x_0 \in I^\infty$ and set $W = I^\infty \setminus \{x_0\}$. Let $\iii \in I^*$. Note that by the definition of the metric $d_2$, for any $\jjj \in I^*$ and $\hhh \in [\jjj]$ we have
\[
  \dist(\hhh, I^\infty \setminus [\jjj]) = \dist([\jjj], I^\infty \setminus [\jjj]) = 2^{1-|\jjj|} = 2\diam([\jjj]).
\]
So if $\varphi_\iii(I^\infty) = [\jjj]$ then
\begin{align*}
 \dist(\varphi_\iii(x), \varphi_\iii(I^\infty\setminus W)) &= d_2(\varphi_\iii(x), \varphi_\iii(x_0)) \le \diam([\jjj]) \\ &= \tfrac 12 \dist(\varphi_\iii(x),I^\infty \setminus [\jjj])
\end{align*}
which implies that \eqref{eq:properdist} holds with $\lambda_x = 1$ and $M = I^\infty$. Using this observation, we can now
give a simple non-Euclidean example of a non-similitude IFS which is properly semiconformal. Let $I = \{0,1,2\}$ and $J=\{1,2\}$. By defining 
\[
  \varphi_1(i\jjj) = \begin{cases}
                      1\jjj & \text{if $i \ne 0$} \\
                      10\jjj & \text{if $i = 0$}
                    \end{cases}, \quad
  \varphi_2(i\jjj) = \begin{cases}
                      2\jjj & \text{if $i  \ne 0$} \\
                      20\jjj & \text{if $i = 0$}
                    \end{cases}
\]
for $i\in I$ and $\jjj \in I^\infty$, we get an IFS $\{\varphi_1,\varphi_2\}$ on $I^\infty$. Given a cylinder $[\iii]$, $\iii \in I^*$, it is clear that $\varphi_j([\iii])$ for $j \in J$ is one of the following cylinders: $[1\iii]$, $[10\iii]$, $[2\iii]$ or $[20\iii]$. It is also easy to see that with any $\jjj \in J^*$ and $\hhh \in I^\infty$ we have either $\varphi_\jjj(\hhh) = \jjj\hhh$ or $\varphi_\jjj(\hhh) = \jjj0\hhh$. This gives
\[
  2^{-|\jjj| - 1} d_2(\hhh,\kkk) \le d_2(\varphi_\jjj(\hhh),\varphi_\jjj(\kkk)) \le 2^{-|\jjj|} d_2(\hhh,\kkk)
\]
for $\jjj \in J^*$ and $\hhh,\kkk \in I^\infty$, establishing the semiconformality of the IFS. Moreover, since both $\varphi_1$ and $\varphi_2$ map cylinders to cylinders, the IFS in this example is properly semiconformal.
\end{example}

The following theorem was proved for the properly Euclidean case in \cite[Corollary 5.8]{KV2008}. In \cite[Remark 6.2]{BR2007} it was suggested that the generalization to doubling metric spaces could be done by extending the thermodynamical formalism \cite{F1997} to that setting. The proof given here uses the more direct Moran construction approach. 

\begin{theorem}\label{thm:equiv}
For a properly semiconformal IFS in a complete doubling metric space the following conditions are equivalent:
\begin{enumerate}
 \item The ball condition.
 \item $\HH^t(E) > 0$ with $P(t) = 0$.
 \item The open set condition.
 \item The strong open set condition.
\end{enumerate}
\end{theorem}
\begin{proof}
The equivalence of (1) and (2) has been established in Corollary \ref{cor:BallCondEquivPressureCond}. In Remark \ref{rmk:semiconf}(i) it was noted that under the given assumptions, (3) implies (1). Clearly (4) implies (3). To complete the proof, it is thus enough show that (1) implies (4).

Assume that (1) holds. 
Let $\delta > 0$ and $x \in E$ be from Proposition \ref{prop:wegotballs} and $D \ge 1$ from the definition on semiconformality. Then by \eqref{eq:alwaysright} we get 
\[
 \varphi_\iii(B(x, D^{-1}\delta  \diam(E))) \subset B(\varphi_\iii(x), D^{-1}\delta \overline{s}_\iii \diam(E))
\]
\[
\subset B(\varphi_\iii(x), \delta \underline{s}_\iii \diam(E)) \subset B(\varphi_\iii(x),\delta \diam(\varphi_\iii(E)))
\]
for every $\iii \in I^*$. Therefore from Proposition \ref{prop:wegotballs} we get
\[
 \varphi_\iii(B(x, D^{-1}\delta  \diam(E))) \cap \varphi_\jjj(B(x, D^{-1}\delta  \diam(E))) = \emptyset
\]
whenever $\iii \perp \jjj$. Clearly $x \in E \cap B(x, D^{-1}\delta  \diam(E))$. Thus the IFS satisfies the SOSC, and the proof is finished.
%
\end{proof}

For the rest of this section we let $M$ be any complete metric space. In this setting the OSC ceases to imply any bounds on the size of the invariant set. As shown in \cite[Example 3.1]{S1996}, the invariant set of a similitude IFS in a complete metric space might consist of a single point, even when the OSC is satisfied. The SOSC, however, continues to be relevant in the general setting. To show this, we first recall a useful result by K. Falconer. An IFS is said to satisfy the \emph{strong separation condition (SSC)} if the images $\varphi_i(E)$,
$i \in I$, are pairwise disjoint for the invariant set $E$.

\begin{proposition}\label{prop:LowerDimBoundFromSSC}
  Let $E$ be the invariant set of an IFS $\{ \varphi_i : i \in I\}$ for which there are constants
$s_i$, $i \in I$, such that
\[
  d(\varphi_i(x), \varphi_i(y)) \ge s_i \,d(x,y)
\]
for $x,y \in M$ and $i\in I$. If the IFS satisfies the SSC, we have $dim_H(E) \ge d$ where
\[
 \sum_{i \in I} s_i^d = 1.
\]
\end{proposition}
\begin{proof}
  Although the proof of this result in \cite[Proposition 9.7]{F1990} is formulated in the Euclidean setting, it remains valid in the general case.
\end{proof}

\begin{lemma}\label{lem:UniformSemiconformalBounds}
  Assuming that constants $\underline{s}_\iii$, $\iii \in I^*$, correspond to a semiconformal IFS (with pressure $P$) by way of \eqref{eq:semiconformality}, there is a constant $C \ge 1$ such that
\[
   C^{-1} \underline{s}_\iii \underline{s}_\jjj \le \underline{s}_{\iii\jjj} \le C \underline{s}_\iii \underline{s}_\jjj
\]
for any $\iii,\jjj \in I^*$ and
\[
 C^{-t} e^{nP(t)} \le \sum_{\iii \in I^n} \underline{s}_\iii^t \le C^{t} e^{nP(t)}
\]
for all $t \ge 0$ and $n \in \N$.
\end{lemma}
\begin{proof}
   Combine \eqref{eq:semiequal}, Proposition \ref{prop:semiIFStractable} and \cite[Lemma 2.1]{KV2008}.
\end{proof}

The first part of the following result was originally shown by A. Schief for self-similar sets on complete metric spaces \cite[Theorem 2.6]{S1996}. The second part makes it clear that in the semiconformal setting, the overlap between the parts $\varphi_i(E)$, $i\in I$, of the invariant set $E$ is negligible, at least in the measure-theoretical sense,
provided that the SOSC holds.

\begin{proposition}\label{prop:SOSCdim}
  Let $E$ be the invariant set of a semiconformal IFS $\{ \varphi_i : i \in I\}$ defined on a complete metric space. If the SOSC holds, then
\begin{enumerate}
 \item[(i)] $\dimh(E) = P^{-1}(0)$.
 \item[(ii)] $\dimh(\varphi_\iii(E) \cap \varphi_\jjj(E)) < \dimh(E)$ whenever $\iii \perp \jjj$.
\end{enumerate}
\end{proposition}
\begin{proof}
  Assume that $U$ is an open set given by the SOSC. Then there exist $x\in U \cap E$ and $\hhh \in I^*$ such that
$x \in E_\hhh \subset U$.

(i) We follow the proof of \cite[Theorem 2.6]{S1996} with appropriate modifications.  Let $k\in \N$. Since the sets $\varphi_{\iii\hhh}(E) \subset \varphi_\iii(U)$ and $\varphi_{\iii'\hhh}(E) \subset \varphi_{\iii'}(U)$ are disjoint for distinct $\iii,\iii' \in I^k$, the IFS $\mathcal{F}_k := \{ \varphi_{\iii\hhh} : \iii \in I^k \}$ satisfies the SSC. Let $F_k$ be the invariant set for $\mathcal{F}_k$ and let $d_k$ be the unique positive number that satisfies
\[
  \sum_{\iii \in I^k} \underline{s}_{\iii\hhh}^{\,d_k} = 1
\]
where the constants $\underline{s}_\iii$, $\iii \in I^*$, are from the definition of semiconformality.
By Proposition \ref{prop:LowerDimBoundFromSSC} we have $\dimh(F_k) \ge d_k$. On the other hand, Lemma \ref{prop:UpperBoundForDimm} gives $\dimh(E) \le \dimm(E) \le P^{-1}(0)$ and clearly $F_k \subset E$, so
\[
  d_k \le \dimh(F_k) \le \dimh(E) \le P^{-1}(0).
\]
Set $t = \dimh(E)$ and $T = P^{-1}(0)$.  The proof of (i) is now completed by showing that we cannot have $t < T$. Apply Lemma \ref{lem:UniformSemiconformalBounds} to get a constant $C \ge 1$ such that
$\underline{s}_{\iii\hhh} \ge C^{-1} \underline{s}_{\iii}\underline{s}_{\hhh}$ for each $\iii \in I^k$ and
\[
   C^{-T} \le \sum_{\iii \in I^k} \underline{s}_\iii^T \le C^T.
\]
Now since $0 < \underline{s}_\iii < 1$ for each $\iii \in I^*$ and $d_k \le t$ for each $k\in \N$, we have
\[
  1 = \sum_{\iii \in I^k} \underline{s}_{\iii\hhh}^{\,d_k} \ge C^{-d_k} \underline{s}_{\hhh}^{\,d_k} \sum_{\iii \in I^k} \underline{s}_{\iii}^{\,d_k} \ge C^{-t} \underline{s}_{\hhh}^{\,d_k} \sum_{\iii \in I^k} \underline{s}_{\iii}^{t},
\]
so by assuming $t<T$ we would get
\begin{align*}
  \underline{s}_\hhh^{-t} &\ge \underline{s}_\hhh^{-d_k} \ge C^{-t} \sum_{\iii \in I^k} \underline{s}_{\iii}^{t} = C^{-t} \sum_{\iii \in I^k} \underline{s}_{\iii}^{T} \underline{s}_{\iii}^{t-T} \\
  &\ge C^{-(t+T)} (\max\nolimits_{\iii \in I^k} \underline{s}_{\iii})^{t-T}
\end{align*}
for any $k\in \N$. However, this contradicts the observation that by Lemma \ref{lem:diamBoundExponential} we have $\lim_{k\to\infty}(\max\nolimits_{\iii \in I^k}\underline{s}_{\iii})^{t-T} = \infty$  if $t<T$. Thus $t = T$.

(ii) Here we essentially reproduce the proof of \cite[Proposition 4.9]{KV2008}. 
It is easy to see that the set
\[
  A := E_\hhh \cup \bigcup_{\kkk \in I^*} E_{\kkk\hhh}
\]
satisfies $\varphi_\iii(A) \cap \varphi_\jjj(A) = \emptyset$ whenever $\iii \perp \jjj$. Therefore
\[
  E_\iii \cap E_\jjj \subset \varphi_\iii(E\setminus A) \cup \varphi_\jjj(E\setminus A)
\]
whenever $\iii \perp \jjj$. The bi-Lipschitz mappings $\varphi_\iii$, $\iii \in I^*$, preserve the dimension, so it is now enough to show that $\dimh(E\setminus A) < \dimh(E)$.

Let $F$ be the invariant set of the semiconformal IFS 
$\{ \varphi_\kkk : \kkk \in J_0 \}$ where $J_0 = I^{|h|}\setminus\{\hhh\}$. It is evident that $E\setminus A \subset F$.
Set $m = |\hhh|$, let $J = I^m$, let $P_J$ and $P_{J_0}$ be the pressures of $\{ \varphi_\kkk \}_{\kkk \in J}$ and $\{ \varphi_\kkk \}_{\kkk \in J_0}$, respectively, and let $u = P_J^{-1}(0)$. Recalling that by Lemma \ref{lem:UniformSemiconformalBounds} we have a constant $C \ge 1$ such that $C^{-1} \underline{s}_\iii \underline{s}_\jjj \le \underline{s}_{\iii\jjj} \le C \underline{s}_\iii \underline{s}_\jjj$ for all $\iii,\jjj \in I^*$ and it further holds that $\max_{\iii\in I^n} \underline{s}_\iii \to 0$ as $n \to \infty$, we can apply \cite[Lemma 2.4]{KV2008} to infer that $P_{J_0}(u) < 0$. Thus $P_{J_0}^{-1}(0) < P_J^{-1}(0)$ by Lemma 2.5. On the other hand,
\begin{align*}
   0 = \frac 1m P_J(u) &= \lim_{n \to \infty} \frac 1{mn} \log \sum_{\kkk \in J^n} \underline{s}_\kkk^u
                    = \lim_{n \to \infty} \frac 1{mn} \log \sum_{\iii \in I^{mn}} \underline{s}_\iii^u
                    = P(u)
\end{align*}
which shows that $P_J^{-1}(0) = u = P^{-1}(0) = \dimh(E)$. By Proposition \ref{prop:UpperBoundForDimm}
we now have
\[
\dimh(E\setminus A) \le \dimm(F) \le P_{J_0}^{-1}(0) < P_{J}^{-1}(0) = \dimh(E)
\]
and the proof is complete.
\end{proof}

We end this section by uncovering  a natural topological prerequisite for the validity of the dimension formula $\dimh(E) = P^{-1}(0)$ when $E$ is the invariant set of a semiconformal IFS. The result shows, in particular, that in the semiconformal setting the overlap between the parts $\varphi_i(E)$, $i \in I$, is  insignificant also in the topological sense if the SOSC holds.

\begin{proposition}\label{lem:NowhereDenseIntersection}
  Let $\{ \varphi_i \}_{i\in I}$ be a semiconformal IFS with pressure $P$ and invariant set $E$ such that $\dimh(E) = P^{-1}(0)$. Then $\varphi_\iii(E) \cap \varphi_\jjj(E)$ is nowhere dense in $E$ whenever $\iii \perp \jjj$.
\end{proposition}
\begin{proof}
  Assume that $\iii \perp \jjj$. It is to be proved that there are no balls $B(x,r)$ with $x \in \varphi_\iii(E) \cap \varphi_\jjj(E) $ and $r>0$ such that $B(x,r) \cap E \subset \varphi_\iii(E) \cap \varphi_\jjj(E)$. Assume, to the contrary, that such a ball $B(x,r)$ exists. Then $x = \pi(\hhh)$ for some $\hhh \in I^\infty$ starting with $\iii$. Now by taking a sufficiently large $m \in \N$ we get $m > |\jjj|$ and
\[
  \varphi_{\hhh|_m}(E) \subset B(x,r) \subset \varphi_\jjj(E) = \bigcup_{\iii \in I^{m-|\jjj|}} \varphi_{\jjj\iii}(E),
\]
from which we infer that
\[
  E = \bigcup_{\substack{\kkk \in I^m\\ \kkk \ne \hhh|_m}} \varphi_{\kkk}(E).
\]
Thus $E$ is also the invariant set of the semiconformal IFS $\mathcal{F} := \{ \varphi_\kkk : \kkk \in J_0 \}$ where $J_0 = I^m\setminus\{\hhh|_m\}$. Denoting the pressure of $\mathcal{F}$ by $P_{J_0}$, we should now have $\dimh(E) \le P_{J_0}^{-1}(0)$. However, as we showed in the proof of the second part of Proposition \ref{prop:SOSCdim}, $P_{J_0}^{-1}(0)$ is strictly smaller than
$\dimh(E)$. This contradiction finishes the proof.
\end{proof}


\section{Sub-constructions}\label{section:sub}

In sections 2 and 3 we studied the generalization of controlled Moran constructions in the direction of weakly controlled Moran constructions. This meant, in particular, that we used the whole space of words $I^\infty$ and relaxed the requirement on the compact sets by replacing the condition \ref{cond:M2} with conditions \ref{cond:cover} and \ref{cond:doub}. There is another natural way to generalize controlled Moran constructions. That is to consider suitable sub\-sets of $I^\infty$.

It is clear that the projection of an arbitrary subset of $I^\infty$ can be geometrically extremely bad. For our purpose we impose a very strict condition on these subsets. This will give a simple way of constructing sets of desired Hausdorff dimension. The example in the Carnot groups we present at the end was the motivation for the following definition.

Suppose we have a compact set $J \subset I^\infty$ and a collection $\{X_\iii \subset M : \iii \in J_*\}$ of compact sets with positive diameter. We write $J_n := \{\iii \in I^n : [\iii] \cap J \ne \emptyset\}$ and $J_* := \bigcup_{n=1}^\infty J_n$. The collection $\{X_\iii : \iii \in J_*\}$ is to be called a \emph{$t$-controlled Moran sub-construction ($t$-CMSC)}, with $t>0$, provided that conditions \ref{cond:incl} and \ref{cond:decr} are satisfied and the following holds:
There exists a constant $C > 0$ so that for every $\iii \in J_*$ and $n \in \N$
\begin{equation}\label{cond:regular}
 C^{-1} \diam(X_\iii)^t < \sum_{\substack{\jjj \in I^n\\ \iii\jjj\in J_*}}\diam(X_{\iii\jjj})^t < C  \diam(X_\iii)^t.
\end{equation}
The set $E = \pi(J)$ is then called the limit set of the CMSC. Notice the relation between the condition \eqref{cond:regular} and the condition \ref{cond:M2} in the definition of a CMC.

\begin{example}
Let us consider sub-constructions of a $\frac{1}{3}$-Cantor set on the real line. Take $I = \{1,2\}$, $f_1(x) = x/3$ and $f_2(x) = x/3+2/3$, and define $X_\iii = f_\iii([0,1])$. The standard $\frac{1}{3}$-Cantor set $C_{1/3}$ is then the limit set of the CMC $\{X_\iii : \iii \in I^*\}$. For it we have $0 < \HH^{s}(C_{1/3}) < \infty$ with $s = \frac{\log 2}{\log 3}$. Now for any $0 < t < s$ we can make a $t$-CMSC for example in the following way:

Let $j_1 = 2$. For $i \ge 1$ define $j_{i+1} = 1$ if $(\prod_{l=1}^{i}j_l)3^{-tl}>1$, and $j_{i+1} = 2$ otherwise. Let $J = \{1, \dots, j_1\} \times \{1, \dots, j_2\} \times \cdots.$ Now for every $\iii \in J_*$ and $n \in \N$
\[
 \sum_{\substack{\jjj \in I^n\\ \iii\jjj\in J_*}}\diam(X_{\iii\jjj})^t = 3^{-tl}\prod_{l=1}^{n}j_{|\iii|+l} \in \left[\frac{1}{4}\diam(X_\iii)^t, 4\diam(X_\iii)^t\right],
\]
and so $\{X_\iii : \iii \in J_*\}$ is a $t$-CMSC.

\end{example}

\begin{proposition}\label{prop:subconstr}
Suppose that we have a $t$-CMSC. Then $\HH^t(E) < \infty$. If the CMSC satisfies the finite clustering property, then $\HH^t(E) > 0$.
\end{proposition}
\begin{proof}
The first claim follows immediately by noticing from \ref{cond:decr} that we can use $\{X_\iii$, $\iii \in J_n\}$ as a cover when estimating the Hausdorff measure of $E$.

Let us prove the second claim. For this it is enough to prove that $M^t(J) > 0$. The rest will follow as in the proof of Proposition \ref{prop:fctoposfinite}. Because $J$ is compact it is enough to look at finite covers. Let $Q \subset J_*$ be finite so that $J \subset \bigcup_{\iii \in Q} [\iii]$ and $[\iii] \cap [\jjj] = \emptyset$ for $\iii, \jjj \in Q$ with $\iii \ne \jjj$. Define $m = \max\{|\iii| : \iii \in Q\}$. Now from the condition \eqref{cond:regular} we get
\begin{align*}
 \sum_{\iii \in Q}\diam(X_\iii)^t & \ge C^{-1}\sum_{\iii \in Q} \sum_{\iii\jjj \in J_{m+1}} \diam(X_{\iii\jjj})^t\\
& = C^{-1} \sum_{\jjj \in J_{m+1}} \diam(X_\jjj)^t \ge C^{-2} \sum_{\jjj \in J_{1}} \diam(X_\jjj)^t
\end{align*}
giving the claim.
\end{proof}

\subsection{An example in Carnot groups}

In \cite{BTW2009} Z. Balogh, J. Tyson and B. Warhurst studied Hausdorff dimensions of sets in Carnot groups. They gave the following comprehensive answer to what the Hausdorff dimensions can be with respect to Carnot-Carath\'eodory and Euclidean metrics.

\begin{theorem}\cite[Theorem 2.4]{BTW2009}\label{thm:BTW}
 In any Carnot group $\G$, we have
\[
 \beta_-(\dim_ES) \le \dim_{cc}S \le \beta_+(\dim_ES)
\]
for every $S \subset \G$.
\end{theorem}

Here $\beta_-$ and $\beta_+$ are the lower and upper dimension comparison functions for $\G$, which will be defined later. The sharpness of the first inequality in Theorem \ref{thm:BTW} was established by using a set of self-similar examples, see \cite[Theorem 4.8]{BTW2009}. The answer was not completely satisfying as the construction worked only for a dense set of dimensions and only for those dimensions gave the answer on the level of positive and finite measures. We will construct the missing compact sets by combining two constructions of the type used in \cite[Proposition 4.14]{BTW2009}. Formally the modification on their construction is a replacement of self-similar construction with a CMSC. Some of the calculations will be omitted and they can be found from \cite{BTW2009}.

We will use the notation from \cite{BTW2009}, but for the convenience we shall recall here some of it. Let $(\G, *)$ be a step $s$ Carnot group with stratified Lie algebra $\mathfrak{g} = \mathfrak{v}_1 \oplus \cdots \oplus \mathfrak{v}_s$, where $[\mathfrak{v}_1, \mathfrak{v}_j]= \mathfrak{v}_{j+1}$ for $j = 1, \dots, s-1$ and $[\mathfrak{v}_1, \mathfrak{v}_s]= 0$. Denote $m_j = \dim \mathfrak{v}_j$. The dilations $\delta_r$ of $\mathfrak{g}$ for $r > 0$ are given by
\[
 \delta_r\left(\sum_{j=1}^s U_j\right) = \sum_{j=1}^s r^jU_j
\]
with $U_j \in \mathfrak{v}_j$. The corresponding dilations on $\G$ are also denoted $\delta_r$.

We use the exponential coordinates in $\G$ which are formed using the exponential map $\exp \colon\mathfrak{g} \to \G$ and a graded orthonormal basis $\{E_{jk} : j= 1, \dots, s; k = 1, \dots, m_j\}$ of $\mathfrak{g}$ by identifying a point $(x_1, \dots, x_s)\in \R^{m_1} \times \cdots \times \R^{m_s}$ with 
\[
\exp\left(\sum_{j=1}^s\sum_{k=1}^{m_j}\langle x_j,e_{jk}\rangle E_{jk}\right),
\]
where $\{e_{jk}\}_{k=1}^{m_j}$ is the standard orthonormal basis of $\R^{m_j}$. With these coordinates we can view the space $\R^{\sum_{j=1}^sm_j}$ with appropriate group operation as our group $\G$. The projections $\pi_j : \G \to \R^{m_j}$ are given by the exponential coordinates as $\pi_j(x_1, \dots, x_s) = x_j$. We also write $\Pi_l = \pi_1 \times \cdots \times \pi_l \colon \G \to \R^{\sum_{j=1}^lm_j}$.

Denote by $d_{cc}$ the Carnot-Carath\'{e}odory metric (see for example \cite{BTW2009} for a definition) and by $d_e$ the Euclidean metric. Instead of the metric $d_{cc}$ we could use any sub-Riemannian metric on the group $\G$ which is left invariant and compatible with the dilations.

Define the \emph{lower dimension comparison function of $\G$} as
\[
 \beta_-(\alpha) = \sum_{j=1}^{l_-}jm_j + (1+ l_-)\left(\alpha - \sum_{j=1}^{l_-}m_j\right)
\]
for $\alpha \in \left]0, \sum_{j=1}^sm_j\right]$, with $l_- \in \{0, \dots, s-1\}$ so that
\[
 \sum_{j=1}^{l_-}m_j < \alpha \le \sum_{j=1}^{1+l_-}m_j.
\]
The \emph{upper dimension comparison function for $\G$} is defined as
\[
 \beta_+(\alpha) = \sum_{j=l_+}^{s}jm_j + (-1+ l_+)\left(\alpha - \sum_{j=l_+}^{s}m_j\right)
\]
for $\alpha \in \left]0, \sum_{j=1}^sm_j\right]$ with $l_- \in \{0, \dots, s-1\}$ so that
\[
 \sum_{j=l_+}^{s}m_j < \alpha \le \sum_{j=-1+l_+}^{s}m_j.
\]

Now we are ready to start with the construction which answers the Remarks 4.9 and 4.10 in \cite{BTW2009}.
\begin{theorem}
Let $\G$ be a Carnot group. Then for every $\alpha \in \left]0, \dim_e \G\right]$ there exists a compact set $K \subset \G$ with
\[
 0 < \HH_e^\alpha(K) \qquad \text{and} \qquad \HH_{cc}^{\beta_-(\alpha)}(K) < \infty.
\]
\end{theorem}
\begin{proof}
We prove the proposition using the ideas of \cite[Proposition 4.14]{BTW2009}.  Let $l \in \{0, \dots, s-1\}$ so that 
\[
\sum_{j=0}^l m_j < \alpha \le \sum_{j=0}^{l+1} m_j.
\]
Let $A_j = \{0, \dots, 2^j-1\}^{m_j}$ for each $j = 1, \dots, s$. Define
\[
\mathcal{F}_1 = \{F_{a_1\cdots a_l} : a_1\in A_1, \dots, a_l \in A_l\} 
\]
and
\[
\mathcal{F}_2 = \{F_{a_1\cdots a_{l+1}} : a_1\in A_1, \dots, a_{l+1}\in A_{l+1}\},
\]
where the functions $F_{a_1\cdots a_k}$ are defined as
\[
 F_{a_1\cdots a_k}(p) = p_{a_1\cdots a_k}*\delta_{1/2}\left(p_{a_1\cdots a_k}^{-1}*p\right)
\]
with $p_{a_1\cdots a_k} = (a_1, \dots, a_k, 0, \dots, 0)$.

Next we define a sequence $(n_i)_{i \in \N} \subset \{1,2\}^\N$ which tells us what system of functions will be used at step $i$. Let $n_1 = 2$ and define the rest by induction as follows: Assume that $n_1, \dots, n_t$ have been defined. Then $n_{t+1} = 2$ if
 \[
\prod_{i = 1}^t n_i^{(l+1)m_{l+1}} < 2^{t(l+1)\left(\alpha - \sum_{j=0}^lm_j\right)}.
\]
Otherwise define $n_{t+1} = 1$.

Let $E$ be the attractor of $\mathcal{F}_2$.  Write $\mathcal{F}_2 = \{g_1, \dots, g_{2^{\sum_{j=0}^{l+1}jm_j}}\}$ with $g_t \notin \mathcal{F}_1$ for $2^{\sum_{j=0}^ljm_j} < t \le 2^{\sum_{j=0}^{l+1}jm_j}$. Write $I = \{1, \dots, 2^{\sum_{j=0}^{l+1}jm_j}\}$. With this enumeration define for $\iii = (i_1, \dots, i_t)\in I^t$ 
\[
X_\iii = g_{i_1}\circ \cdots \circ g_{i_t}(E)
\]
By Proposition \ref{prop:semiIFStractable} the collection $\{X_\iii : \iii \in I^*\}$ is a CMC. Let now
\[
 J = \left\{\iii = (i_1, i_2,  \dots) : i_j \in \N, 1 \le i_j \le n_j^{(l+1)M_{l+1}}2^{\sum_{j=0}^ljm_j}\right\}.
\]
The collection $\{X_\iii : \iii \in J_*\}$ is then a $\beta_-(\alpha)$-CMSC: Let $\iii \in J_*$ and $n > |\iii|$. Then by
\[
 \sum_{\iii\jjj \in J_n} \diam_{cc}(X_{\iii\jjj})^{\beta_-(\alpha)} = \left(2^{|\iii|-n}\diam_{cc}(X_\iii)\right)^{\beta_-(\alpha)}\prod_{i = |\iii|+1}^n n_i^{(l+1)m_{l+1}}2^{t\sum_{j=0}^ljm_j}
\]
and the definition of the sequence $(n_j)_{j=1}^\infty$ we get
\[ 
C^{-1}\diam_{cc}(X_\iii)^{\beta_-(\alpha)} \le \sum_{\iii\jjj \in J_n}\diam_{cc}(X_{\iii\jjj})^{\beta_-(\alpha)}\le C\diam_{cc}(X_\iii)^{\beta_-(\alpha)},
\]
where $C = 2^{2(l+1)m_{l+1}}$. Therefore by Proposition \ref{prop:subconstr} we have $\HH_{cc}^{\beta_-(\alpha)}(K) < \infty$, where $K$ is the limit set of the sub-construction.

To see that $0 < \HH_e^\alpha(K)$ we estimate the level sets of a Lipschitz mapping as in \cite{BTW2009}, but now they are not translates of an invariant set of a self-similar IFS. With a similar calculation as in the proof of \cite[Lemma 4.16]{BTW2009} we see that for almost every $x \in \Pi_l(K)$ the set $\pi_{l+1}(K\cap \Pi_l^{-1}(x))$ is a Euclidean translate of the limit set $K'$ of the Euclidean construction $\{Y_\iii : \iii \in J'_*\}$ with
\[
 Y_\iii = h_{i_1} \circ \cdots \circ h_{i_m}([0,2]^{m_{l+1}}),
\]
$ J' = \{\iii = (i_1, \dots) :  1 \le i_j \le n_j^{(l+1)m_{l+1}}\}$ and $h_j(y) = 2^{-l-1}y + (1- 2^{-l-1})a_j$. Clearly the sub-construction satisfies the finite clustering property. Write $\gamma = \alpha-\sum_{j=0}^lm_j$. The collection $\{Y_\iii : \iii \in J'_*\}$ is now a $\gamma$-CMSC: Let $\iii \in J'_*$ and $n > |\iii|$. Now
\[
 \sum_{\iii\jjj \in J'_n} \diam_e(Y_{\iii\jjj})^\gamma = \left(2^{(l+1)(|i|-n)}\diam_e(Y_\iii)\right)^\gamma \prod_{i = |\iii|+1}^n n_i^{(l+1)m_{l+1}}
\]
gives
\[
 C^{-1}\diam_e(Y_\iii)^\gamma \le \sum_{\iii\jjj \in J'_n} \diam_e(Y_{\iii\jjj})^\gamma\le C\diam_e(Y_\iii)^\gamma
\]
with $C$ as before. Then by Proposition \ref{prop:subconstr} we get $\HH_e^\gamma(K')>0$. Integrating over $\Pi_l(K)$ gives $\HH_e^\alpha(K) > 0$.
\end{proof}

\end{document}